\documentclass[12pt]{amsart}

\usepackage{graphicx}
\usepackage[fontsize=14pt]{scrextend}
\usepackage{amsthm}
\usepackage{url}
\usepackage{subcaption}
\usepackage{caption}
\usepackage{enumitem}
\makeatletter
\renewcommand\@biblabel[1]{}
\makeatother

\newtheorem{theorem}{Theorem}
\newtheorem{prop}[theorem]{Proposition}
\newtheorem{lemma}[theorem]{Lemma}
\newtheorem{cor}[theorem]{Corollary}
\newtheorem{conj}[theorem]{Conjecture}
\theoremstyle{definition}
\newtheorem{definition}[theorem]{Definition}
\newtheorem{example}[theorem]{Example}
\theoremstyle{remark}
\newtheorem{remark}[theorem]{Remark}

\numberwithin{theorem}{section}

\newcommand{\eps}{\varepsilon}
\newcommand{\paren}[1]{\left(#1\right)}

\newcommand{\bracket}[1]{\left[#1\right]}
\newcommand{\abs}[1]{\left\lvert#1\right\rvert}
\newcommand{\mb}{\mathbb}
\newcommand{\del}{\partial}
\newcommand{\pdif}[2]{\frac{\del#1}{\del#2}}
\newcommand{\R}{\mathbb{R}}

\begin{document}

\title{Double Bubbles on the Real Line with Log-Convex Density}

\author[Bongiovanni \emph{et al.}]{Eliot Bongiovanni,  Leonardo Di Giosia, Alejandro Diaz, Jahangir Habib,  Arjun Kakkar, Lea Kenigsberg, Dylanger Pittman,  Nat Sothanaphan, Weitao Zhu}

\begin{abstract}
The classic double bubble theorem says that the least-perimeter way to enclose and separate two prescribed volumes in $\R^N$ is the standard double bubble.
We seek the optimal double bubble in $\R^N$ with density, which we assume to be strictly log-convex. For $N=1$ we show that the solution is sometimes two contiguous intervals and sometimes three contiguous intervals. In higher dimensions we think that the solution is sometimes a standard double bubble and sometimes concentric spheres (e.g. for one volume small and the other large).
\end{abstract}

\maketitle

\setcounter{tocdepth}{1}
\tableofcontents

\section{Introduction}
The double bubble theorem (see \cite[Chapt. 14]{M}) says that in $\R^N$ the standard double bubble of Figure \ref{fig:standarddb},
consisting of three spherical caps meeting at 120 degrees, provides the least-perimeter way to enclose and separate two given volumes.
We want to put density on $\R^N$. We focus on strictly log-convex, $C^1$ radial densities, for which
the sphere about the origin is stable and indeed, for $C^3$ densities, uniquely the best single bubble by Chambers' \cite{Ch} recent proof of the log-convex density conjecture.
In $\R^2$, McGillivray \cite{MG} recently extended Chambers' results to arbitrary (not necessarily smooth) radial log-convex densities.

\begin{figure}[h!]
  \includegraphics[width=0.5\textwidth]{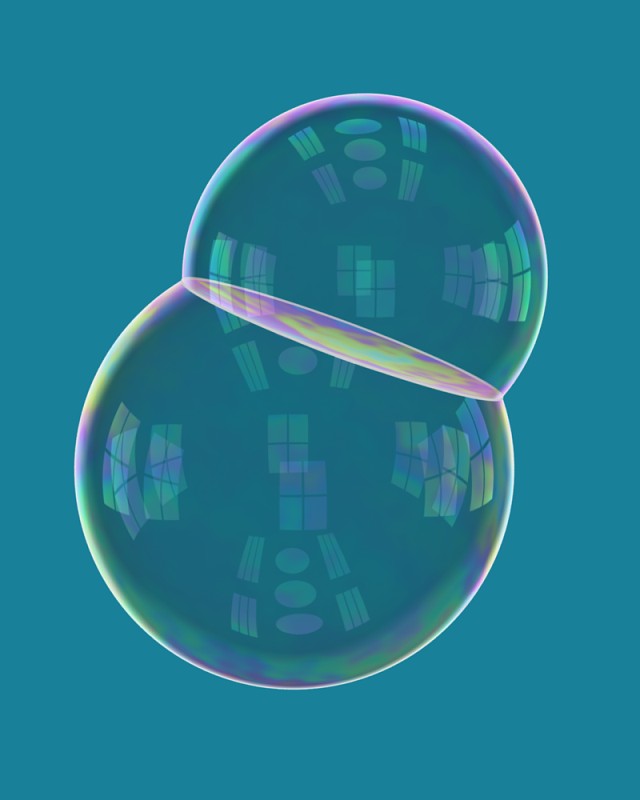}
\caption{The standard double bubble provides the least-perimeter way to enclose and separate two given volumes in $\R^3$ with density 1. We consider more general densities. Image from John M. Sullivan, \protect\url{http://www.math.uiuc.edu/~jms/Images}.}
\label{fig:standarddb}
\end{figure}

For the case of $\R^1$ $(N=1)$, 
we show that there are two types of optimal double bubbles, as illustrated in Figure \ref{fig:1ddbeq}.

\begin{figure}[h!]
  \includegraphics[width=0.7\textwidth]{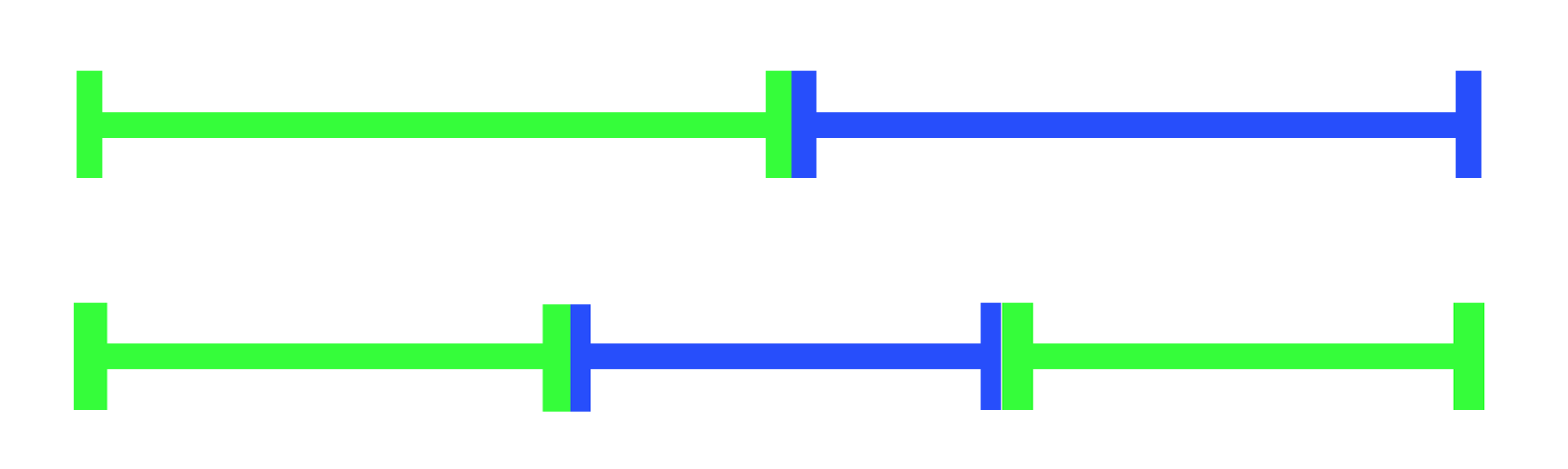}
\caption{Perimeter-minimizing double bubbles
in $\R$ with strictly log-convex $C^1$ symmetric density
can consist of two or three contiguous intervals.}
\label{fig:1ddbeq}
\end{figure}

\begin{prop}[Prop. \ref{lem:db-equalvol}]
On $\R$ with a symmetric, strictly log-convex, $C^1$ density, for equal prescribed volumes, the perimeter-minimizing double bubble is a connected double interval, symmetric about the origin.
\end{prop}

\begin{prop}[Prop. \ref{lem:db-v1fixedv2large}]
On $\R$ with symmetric, strictly log-convex, $C^1$ density $f$ such that $(\log f)'$ is unbounded,
for given $V_1>0$, for sufficiently large $V_2$,
the least-perimeter double bubble is a symmetric interval of volume $V_1$ flanked by two contiguous intervals of volume $V_2/2$.
\end{prop}

\begin{figure}[!h]
  \includegraphics[width=0.49
\textwidth]{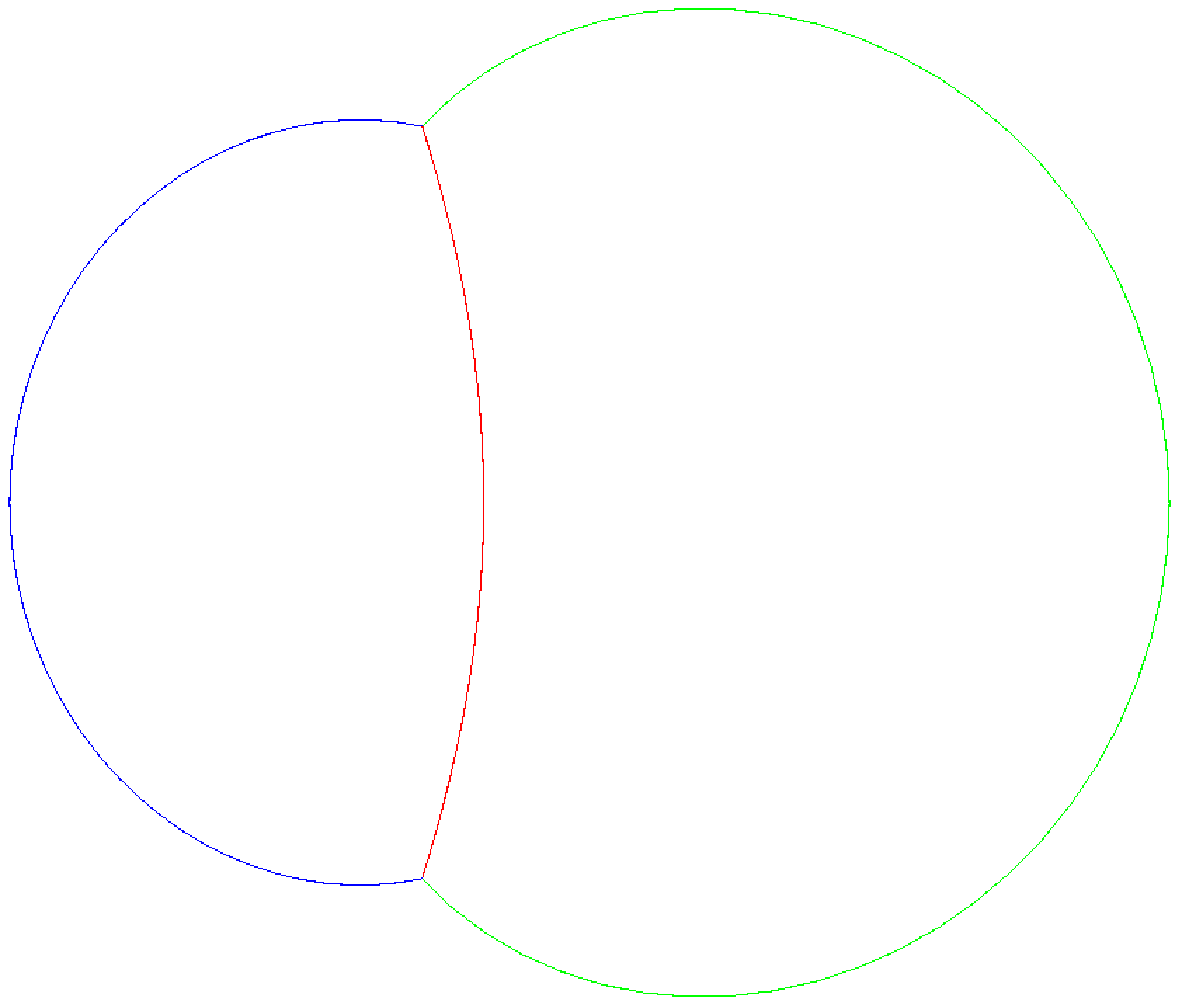}
  \includegraphics[width=0.44
\textwidth]{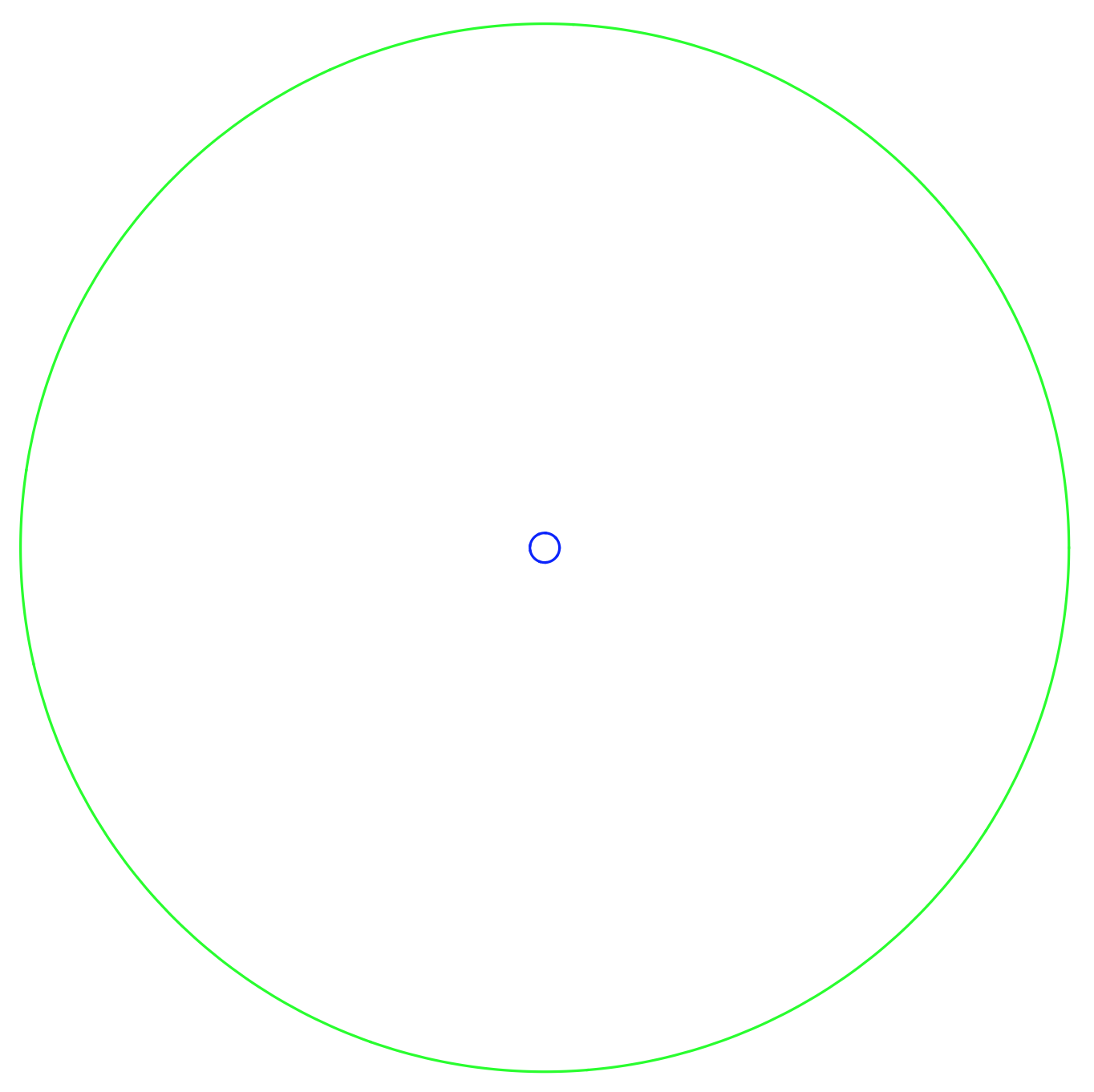}
  \includegraphics[width=0.49
\textwidth]{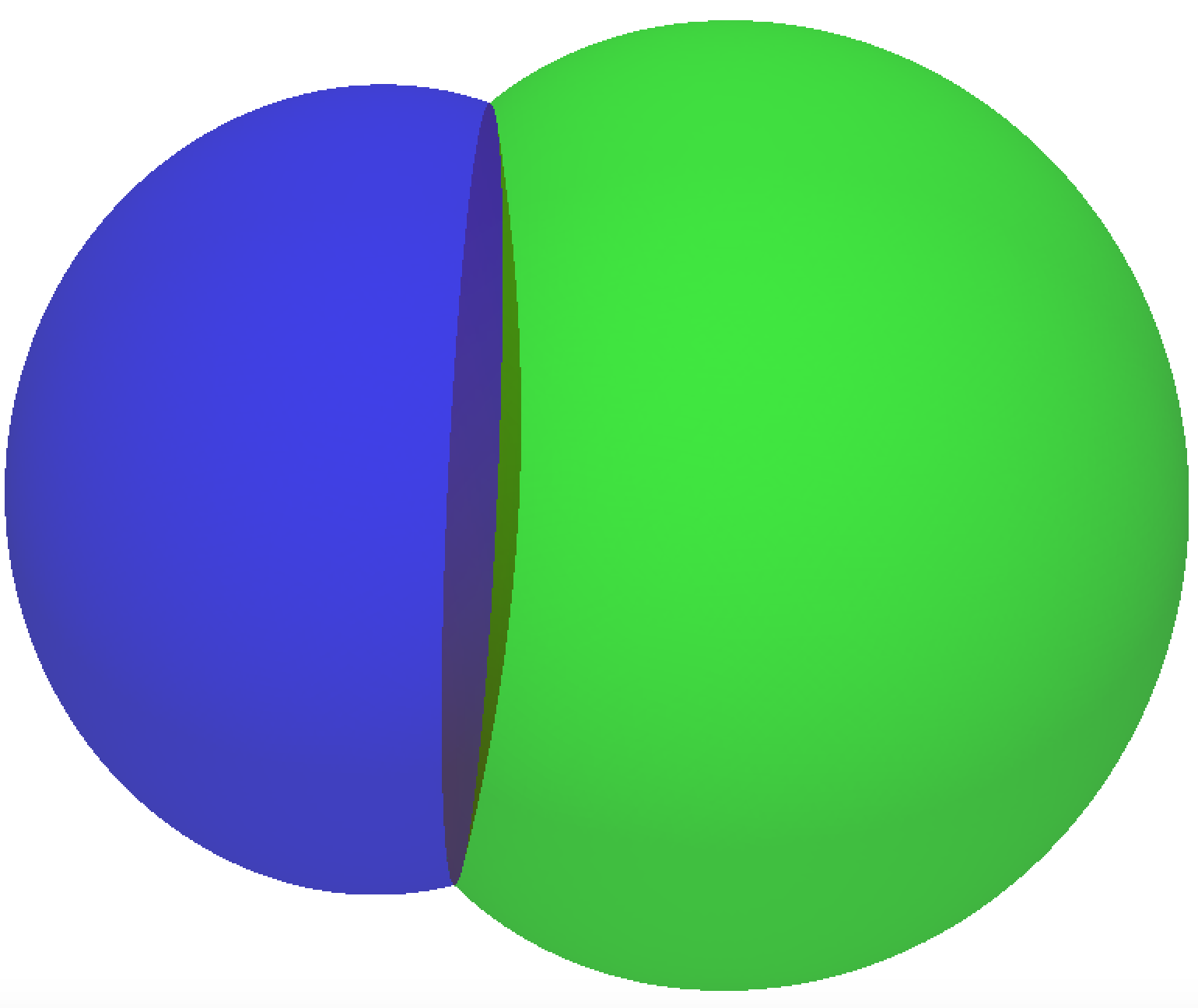}
  \includegraphics[width=0.43
\textwidth]{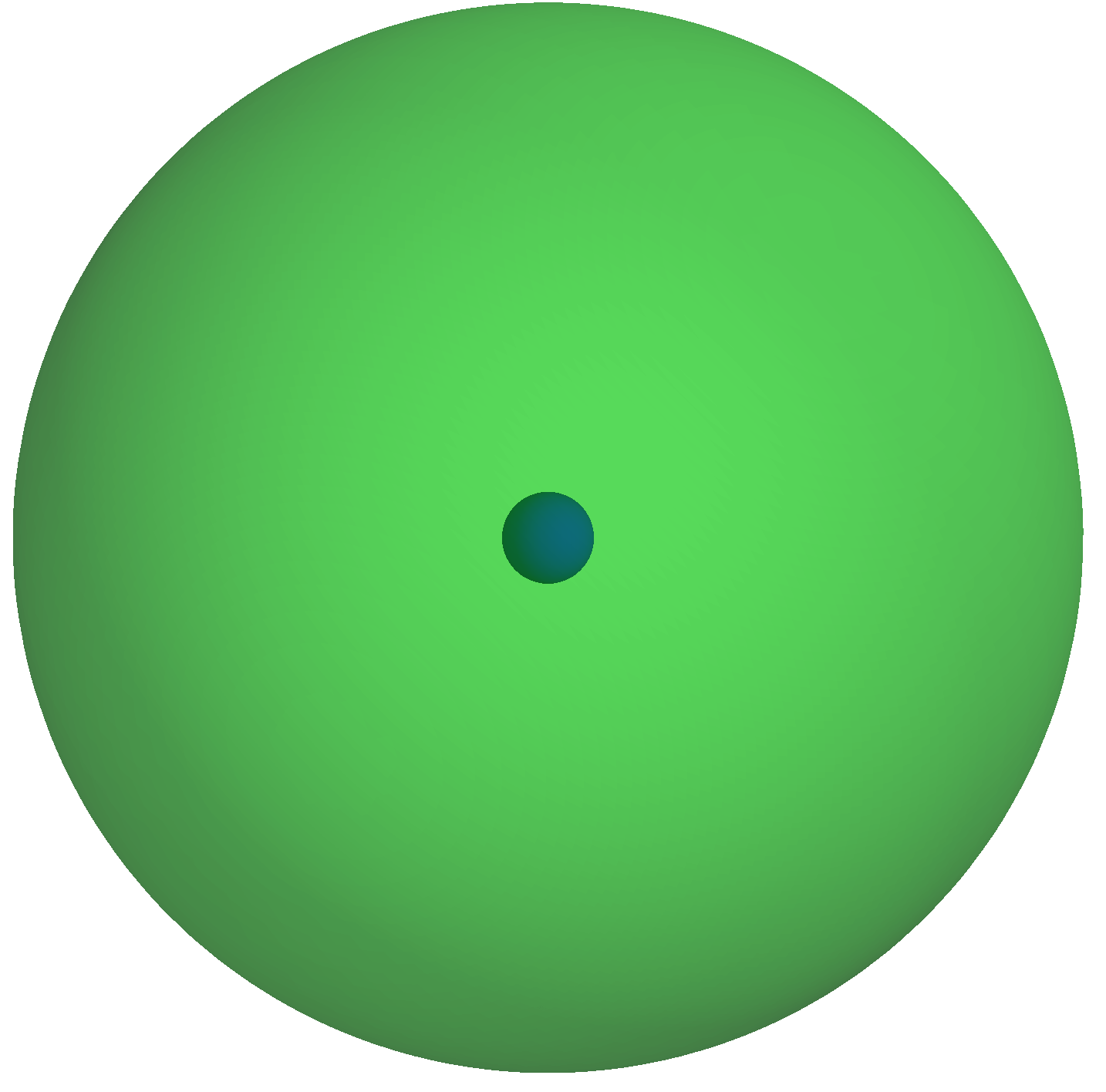}
\caption{We conjecture that a perimeter-minimizing double bubble in $\R^N$ with density $e^{r^2}$ is sometimes a standard double bubble and sometimes a much smaller bubble inside a bubble. Computed with Brakke's Surface Evolver \cite{Br} and Mathematica.}
\label{fig:brakketwothree}
\end{figure}

Our main result characterizes when the optimal double bubble transitions from a double interval to a triple interval, for a strictly log-convex density such that the derivative of the log of the density is unbounded:

\begin{theorem}[Thm. \ref{thm:db-1dmain}]
On $\R$ with a symmetric, strictly log-convex, $C^1$ density $f$ such that $(\log f)'$ is unbounded,
for given $V_1>0$, there is a unique $V_2=\lambda(V_1)$
such that the double interval in equilibrium and the triple interval
tie.
For $V_2>\lambda(V_1)$, the perimeter-minimizing
double bubble is uniquely the triple interval.
For $V_2<\lambda(V_1)$, the perimeter-minimizing
double bubble is uniquely the double interval
in equilibrium.
Moreover, $\lambda$ is a strictly increasing $C^1$ function
that tends to a positive limit as $V_1\to 0$.
\end{theorem}

Section 6 studies the growth rate of the tie curve $\lambda(V_1)$. Our results imply for example that for Borell density $e^{x^2}$, for $V_1$ large, 
$$V_1(\log V_1)^{1/2-\eps} < \lambda(V_1) < V_1^{4+\eps}$$
(Cors. \ref{cor:lowerboundBorelllambda}, \ref{cor:upperboundBorelllambda}).

Our numerics indicated to our astonishment that sometimes as the volumes are scaled up, the minimizer changes from a double interval to a triple interval and then back to a double interval, as in Figure \ref{fig:tieintersection}.

\begin{figure}[!ht]
\includegraphics[width=0.8\textwidth]{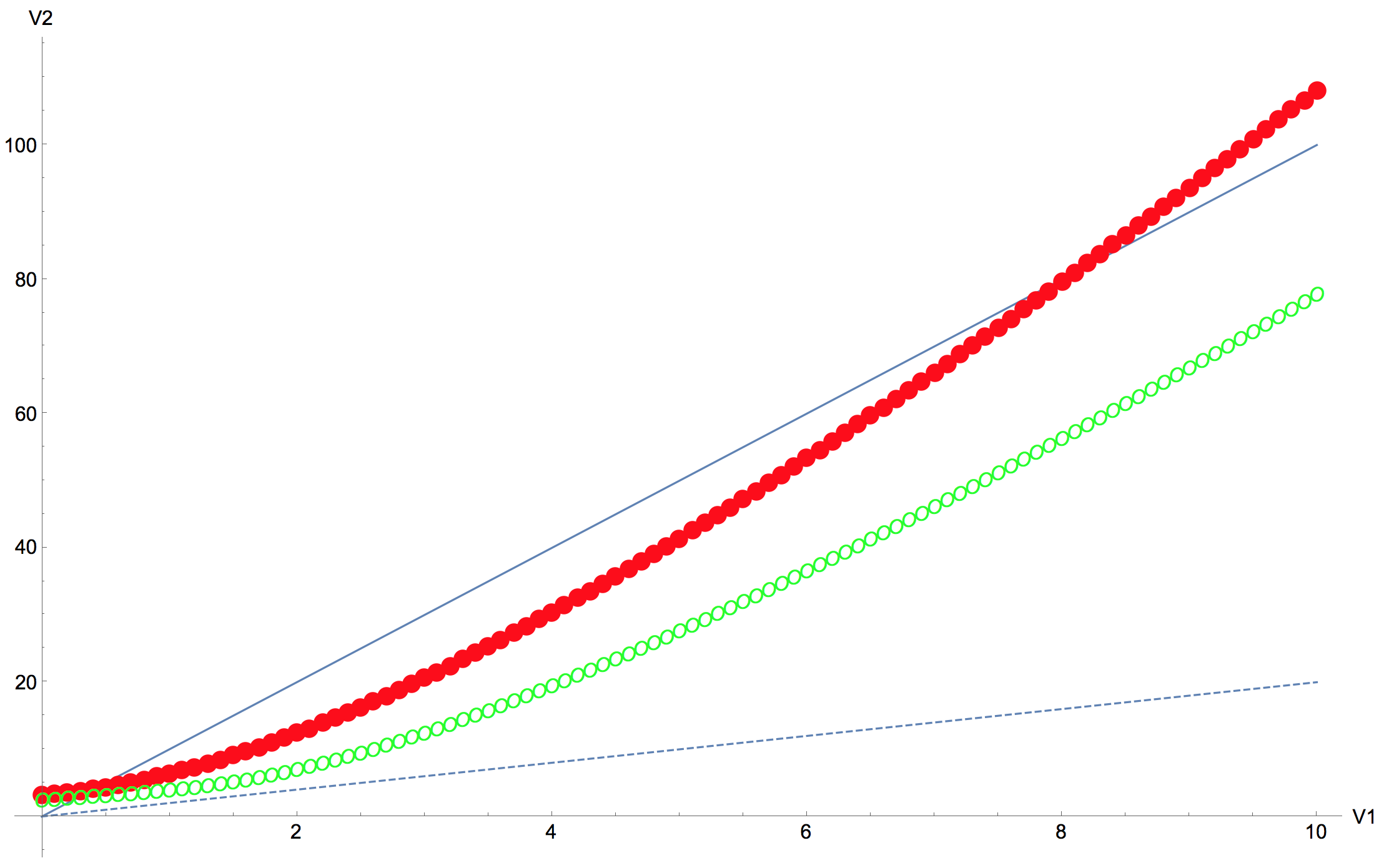}
\caption{This numerically generated plot shows the tie points of the double and triple interval in the curve with solid red points for density $e^{x^2}$ and the curve with the empty green points for density $e^{x^4}$. The solid blue line is $V_2=10V_1$ and the dashed blue line is $V_2=2V_1$. Note that the solid blue line intersects the red curve twice.} 
\label{fig:tieintersection}
\end{figure}

In $\R^N$, we conjecture and provide some numerical evidence that the solution is a standard double bubble (the analog of the double interval) when the sizes are comparable and a bubble inside a bubble (the analog of the triple interval) when one bubble is much larger, as in Figure \ref{fig:brakketwothree}. For equal volumes in 2D, as the volumes increase, the solution tends to a circle centered at the origin
plus diameter (Fig. \ref{fig:g16conj}).

\begin{figure}[h!]
  \includegraphics[width=0.5
\textwidth]{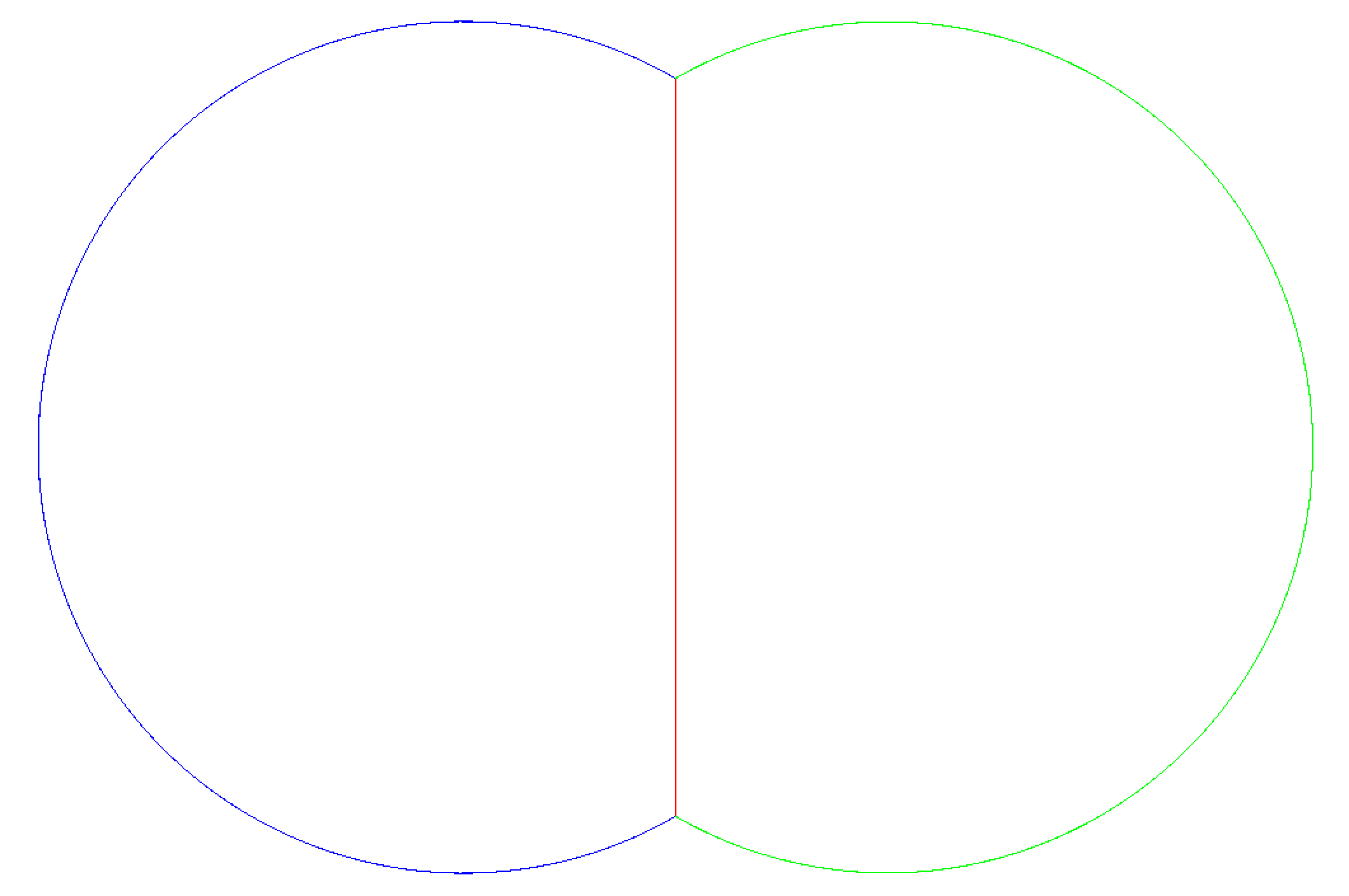}
  \includegraphics[width=0.45
\textwidth]{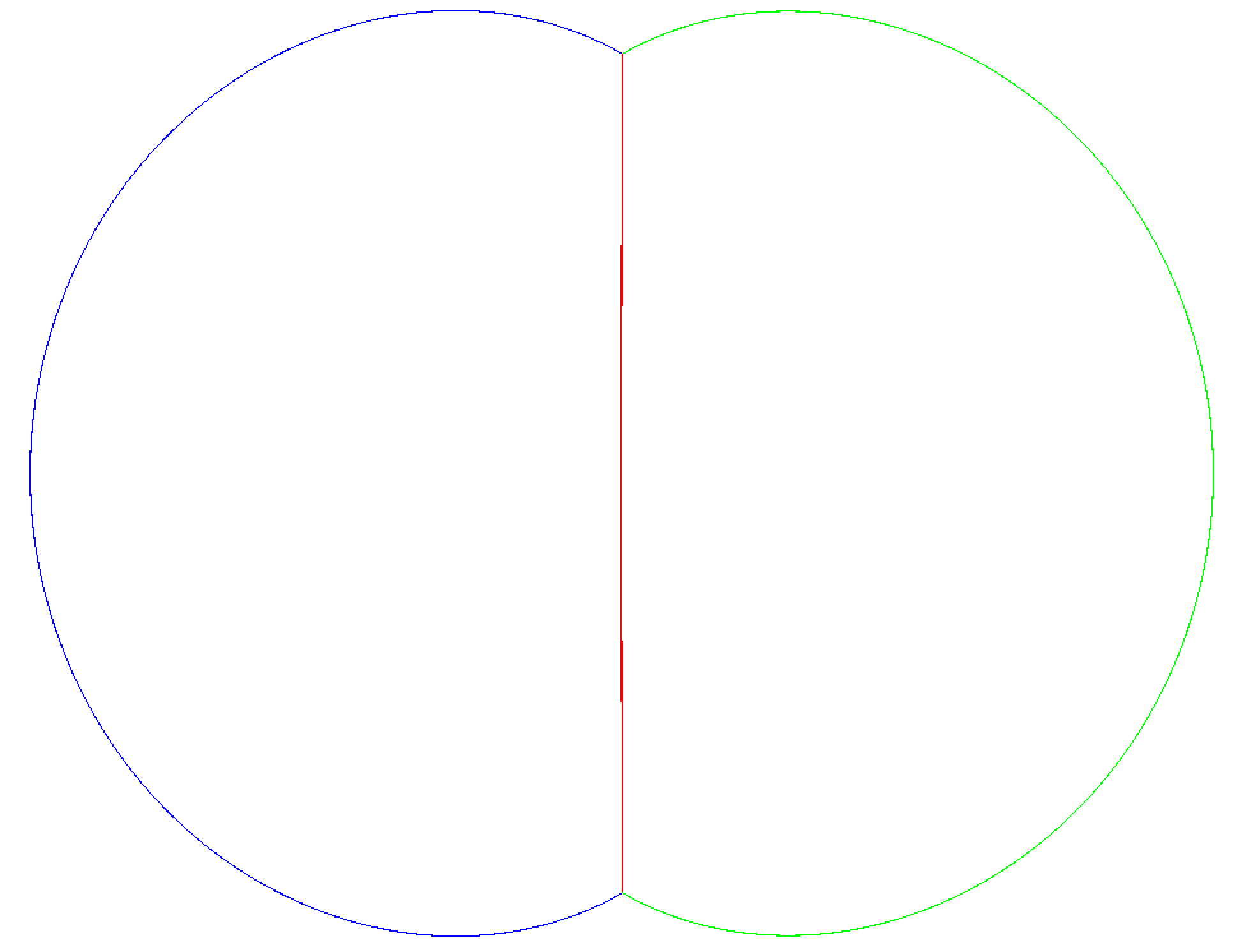}
  \includegraphics[width=0.47
\textwidth]{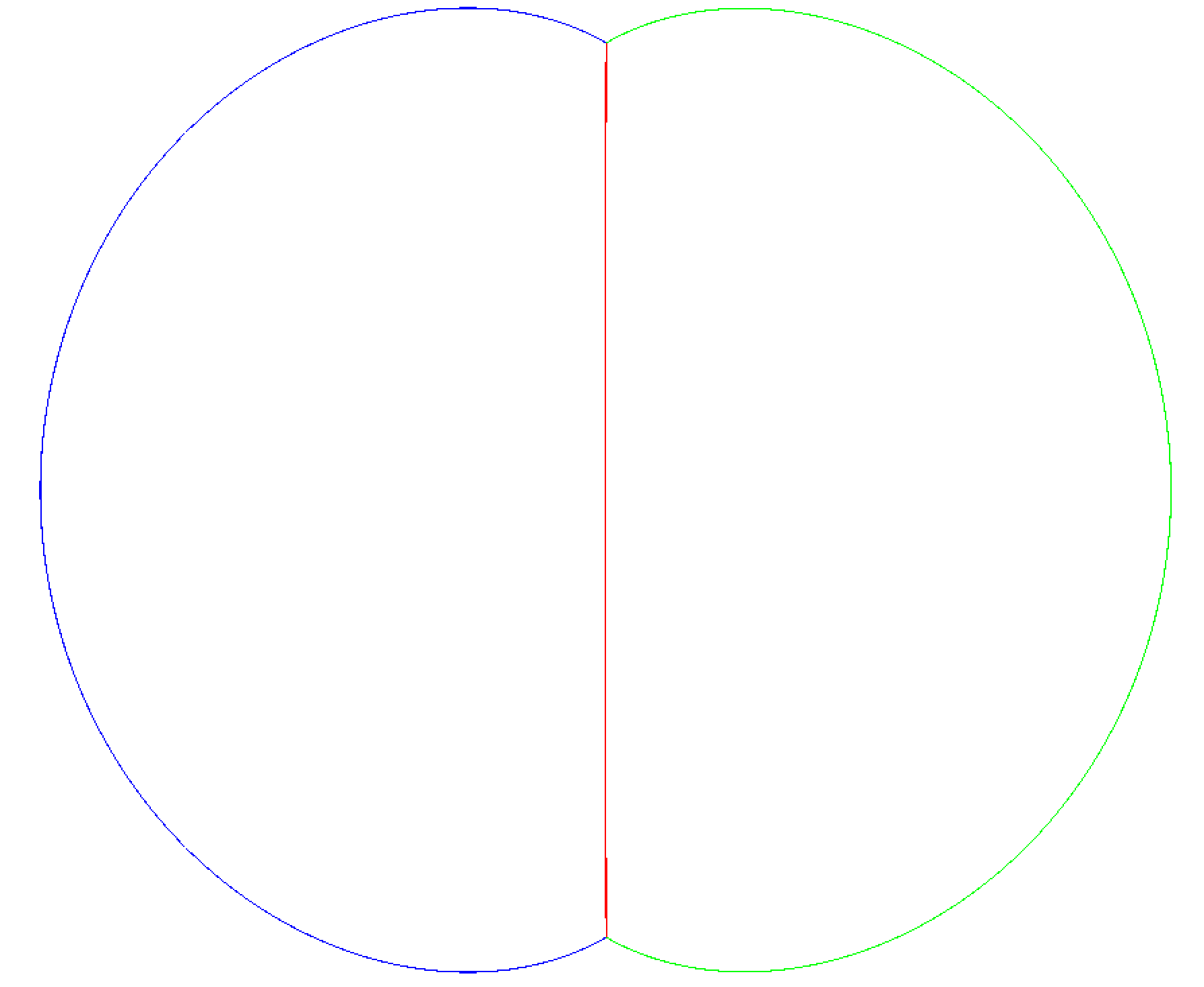}
\includegraphics[width=0.45
\textwidth]{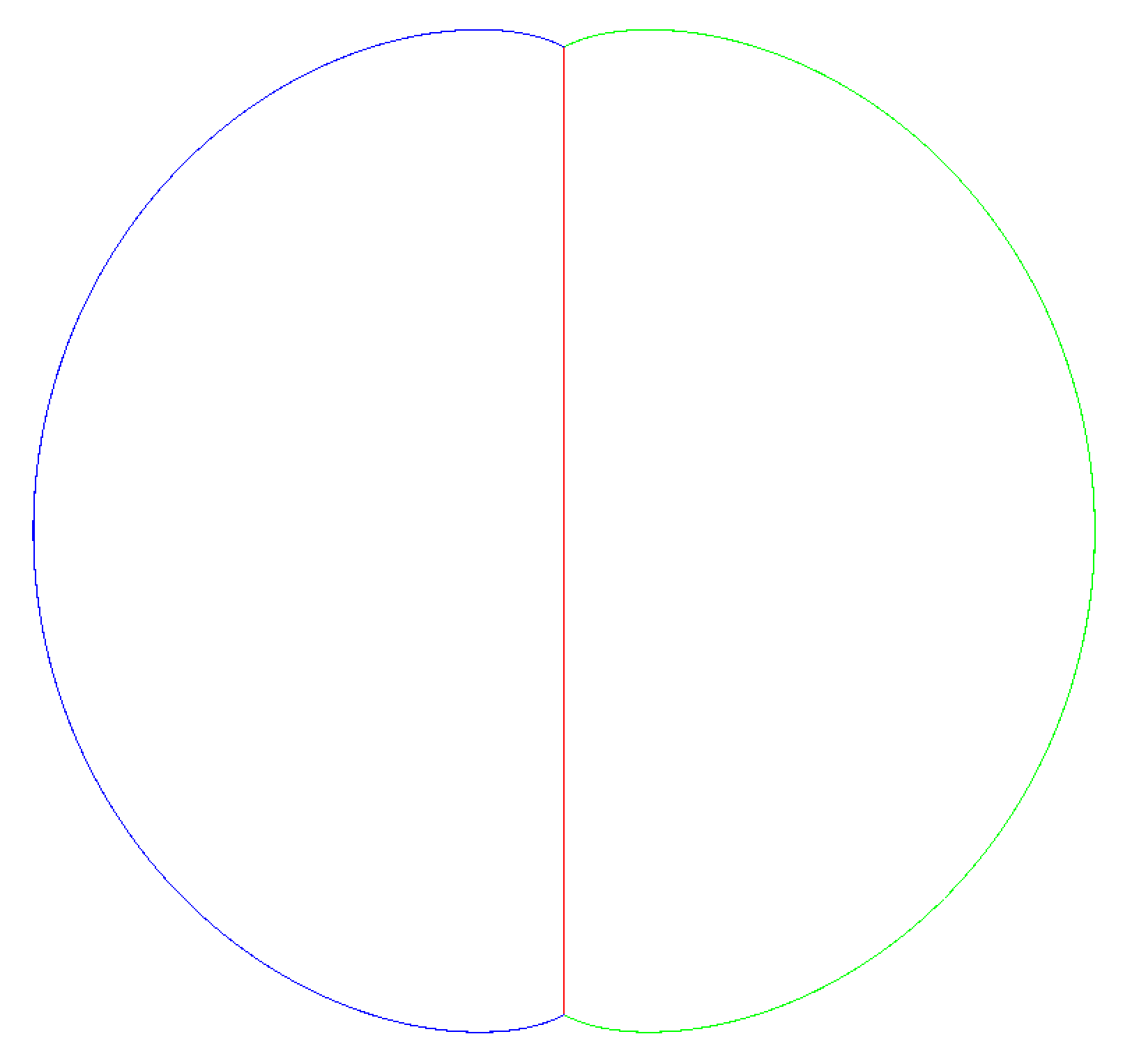}
\caption{For increasing equal volumes (0.01, 0.1, 10 and 1000) the double bubble approaches a circle plus diameter. Computed with Brakke's Surface Evolver.}
\label{fig:g16conj}
\end{figure}

We conjecture that the smoothness assumption of Theorem \ref{thm:db-1dmain} can be omitted. By smoothing, any symmetric strictly log-convex density on $\R$ is a limit of smooth densities. It follows that Proposition \ref{lem:db-equalvol} holds for any symmetric strictly log-convex density. But this argument does not work in general: in Proposition \ref{lem:db-v1fixedv2large}, the threshold for the ``sufficiently large $V_2$'' condition could go to infinity in the limit. Nevertheless, we think that one may be able to obtain the same results by directly working with non-smooth densities via one-sided derivatives.


The triple bubble problem on the real line can be studied with techniques similar to those used in this paper. In fact, we have made some progress, showing that for a symmetric, strictly log-convex density, there are four possible combinatorial types of perimeter-minimizing triple bubbles. Our results on this problem can be found in our report \cite{So1}. However, the triple bubble problem is much more complicated than the double one. The transition boundary is likely many surfaces in $\R^3$ stitched together, making it more difficult to study. Moreover, in the double bubble problem, it happens that there is a single kind of transition (from double to triple intervals) that occurs for every density. We suspect that, in the triple bubble case, there may be different kinds of transitions depending on the density. In particular, we conjecture that only three types of perimeter minimizers occur for some densities and four for others.

The single bubble problem with density was previously studied by Bobkov and Houdr\'{e} \cite{BH} and Bayle \cite{Ba}.
There are results in the literature on double bubbles in the sphere $\mb{S}^N$, hyperbolic space $\mb{H}^N$, flat tori $\mb{T}^2$ and $\mb{T}^3$, and Gauss space (Euclidean space with density $e^{-r^2}$); see \cite[Chapt. 19]{M}. Recently, Milman and Neeman \cite{MN} proved the Gaussian double bubble conjecture, which states that the solution is three halfspaces meeting at 120 degrees.

\subsection*{Outline of proofs}
First we show that a perimeter-minimizing double bubble
consists of two or three contiguous intervals,
by sliding and rearrangement arguments (Prop. \ref{prop:db-2or3}).
Moreover, for fixed $V_1$, as $V_2$ increases from $V_1$, it transitions from double
to triple (Thm. \ref{thm:db-1dmain}).
Our most difficult analysis describes how the transition point $\lambda(V_1)$ increases as $V_1$ increases (Props. \ref{prop:lambdatoinfty} and \ref{prop:lambdaupbd}).

\subsection*{Outline of paper}
Section \ref{sect:notation} defines bubbles and densities.
Section \ref{sect:nbub1d} provides our results on $n$-bubbles on the real line.
Section \ref{sect:db1d} provides our main results on the double bubbles on the real line with strictly log-convex densities.
Section \ref{sect:nonstrict} examines some non-strict log-convex densities
on the real line.
Section \ref{sect:bounds} gives lower and upper bounds on the tie curve
given in Theorem \ref{thm:db-1dmain}.
Section \ref{sect:higher} uses numerical techniques to compute the surface areas of conjectured double bubbles in $\R^2$ and $\R^3$ with Borell density $e^{r^2}$.

\subsection*{Acknowledgements} This paper is a product of the work of the Williams College SMALL NSF REU 2016 and 2017 Geometry Groups. We thank our adviser Frank Morgan for his guidance on this project. We would also like to thank the National Science Foundation; Williams College (including the John and Louise Finnerty Fund); the Mathematical Association of America; Stony Brook University; Michigan State University; University of Maryland, College Park; and Massachusetts Institute of Technology.

\section{Densities and Bubbles}
\label{sect:notation}
\begin{definition}
A \emph{density} on $\mathbb{R}^N$ is just a positive function, used to weight volume and perimeter. A \emph{bubble} in $\mathbb{R}^N$ is a region of prescribed (weighted) volume and perhaps many components. An \emph{$n$-bubble} consists of $n$ bubbles with disjoint interiors, which may or may not share boundaries. A 2-bubble is also called a \emph{double bubble}.
Each shared boundary is counted only once in the perimeter.
An $n$-bubble that minimizes perimeter
for its enclosed volumes is called \emph{perimeter minimizing} or \emph{isoperimetric}.
\end{definition}

\section{$n$-Bubbles on the Real Line}
\label{sect:nbub1d}
We consider $\R$ with density $f$. If $f$ is bounded below and an $n$-bubble has finite weighted perimeter, then each region consists of finitely many intervals.
The boundary points divide $\R$ into closed intervals
(which may be infinite on one side) called \emph{blocks}.
A block may be a component of a bubble, or its interior may not intersect any bubble.

This section contains results on existence (Prop. \ref{prop:1d-exist}),
equilibrium (Cor. \ref{prop:1d-equil}),
and regularity (Prop. \ref{prop:1d-contiguous})
for $n$-bubbles on the real line with density.
Proposition \ref{prop:1d-single} identifies the optimal single bubble as a symmetric interval. Proposition \ref{prop:2n-1} proves that a perimeter-minimizing $n$-bubble has at most $2n-1$ components.

\begin{prop}
\label{prop:1d-exist}
On $\R$ with continuous density $f$ going to infinity in both directions, given $n$ finite volumes $V_i >0$, a perimeter-minimizing $n$-bubble exists and consists of finitely many intervals.
\end{prop}

\begin{proof}
Since $f$ has a positive lower bound, candidates consist of a bounded number of intervals. Since $f$ goes to infinity in both directions, candidates lie in a bounded region.
By compactness, there is a sequence of candidates whose perimeters tend to the infimum and whose endpoints converge. Because $f$ is continuous, the limit of these candidates is an $n$-bubble enclosing the desired volumes and the perimeter is the infimum.
So a perimeter-minimizing $n$-bubble exists (and consists of finitely many intervals).
\end{proof}

\begin{prop}\emph{(First Variation Formula).}
\label{prop:FirstVariation}
Let $f$ be a $C^1$ density on $\R$. Then the first derivative of perimeter moving a point x to the right at rate $1/f$ to alter volume at unit speed is given by
$$\frac{dP}{dt}=(\log f)'(x).$$
\end{prop}

\begin{proof}
$$\frac{dP}{dt} = \frac{dP}{dx} \frac{dx}{dt} = f'\frac{1}{f} = (\log f)'.$$
\end{proof}

\begin{cor}
\label{prop:1d-equil}
Let $f$ be a $C^1$ density on $\R$.
If an $n$-bubble with boundary points $x_1<x_2<\dots<x_k$ is perimeter minimizing, then
$$\sum_{i=1}^k (\log f)'(x_i)=0.$$
More generally, if $1 \leq a < b \leq k$ are such that
the blocks to the left of $x_a$ and to the right of $x_b$ both belong to the same bubble or to no bubble, then
$$\sum_{i=a}^b (\log f)'(x_i)=0.$$
\end{cor}

\begin{proof}
Since moving the points at rate $1/f$ preserves volumes and the $n$-bubble minimizes perimeter for fixed volumes, the derivative $dP/dt$ must vanish. Now the result follows from the First Variation Formula (Prop. \ref{prop:FirstVariation}).
\end{proof}

\begin{remark}
\label{rem:onesidedderiv}
In Corollary \ref{prop:1d-equil}, if the condition on $f$ is relaxed from $C^1$ to one-sided derivatives (for example if $f$ is convex or log-convex), then similarly the sum of the right derivatives is nonnegative and the sum of the left derivatives nonpositive.  
\end{remark}

\begin{prop}
\label{prop:1d-contiguous}
On $\R$ with a continuous density
that is nonincreasing on $(-\infty,0]$ and 
nondecreasing on $[0,\infty)$,
a perimeter-minimizing $n$-bubble
consists of finitely many \emph{contiguous} intervals. 
\end{prop}

\begin{proof}
Because the density is nonincreasing on $(-\infty,0]$ and 
nondecreasing on $[0,\infty)$, it has a positive lower bound. Hence a perimeter-minimizing $n$-bubble consists of
finitely many intervals, or it would have infinite perimeter.

Suppose that these intervals are not contiguous. Then there exist two components $[a,b]$ and $[c,d]$ with $a<b<c<d$, where $(b,c)$ does not intersect any bubble.
We may assume that $b<0$ by symmetry. But then $[a,b]$ can be moved to the right until it reaches $[c,d]$ or the origin so that the volume is preserved and the perimeter does not increase. If $[a,b]$  meets $[c,d]$, two boundary points become one and the total perimeter is less than the original configuration's, contradiction. If $[a,b]$ meets the origin, then $[c,d]$ can be moved to the left while maintaining the volume and reducing the perimeter as before, contradiction.
\end{proof}

For completeness we include a proof of the 1D log-convex density theorem \cite[Cor. 4.12]{RCBM}:

\begin{prop}[Single bubble]
\label{prop:1d-single}
On $\R$ with symmetric, strictly log-convex,
continuous density,
every interval symmetric about the origin is uniquely isoperimetric.
\end{prop}

\begin{proof}
By Proposition \ref{prop:1d-exist}, a perimeter minimizer exists
for a given volume. By Proposition \ref{prop:1d-contiguous},
it is a single interval $[x_1,x_2]$.
Corollary \ref{prop:1d-equil} implies that
$$(\log f)'(x_1)+(\log f)'(x_2)=0$$
for the $C^1$ case. Since $(\log f)'$ is a strictly increasing odd function, 
we have $x_2=-x_1$ and the interval is symmetric about the origin. 

By Remark \ref{rem:onesidedderiv}, for the non-$C^1$ case
\begin{align}
(\log f)_L'(x_1) + (\log f)_L'(x_2) &\leq 0 \label{eq:singleleftder} \\
(\log f)_R'(x_1) + (\log f)_R'(x_2) &\geq 0 \label{eq:singlerightder}
\end{align}
where $(\log f)_L'$ and $(\log f)_R'$ denote the left and right derivatives. Since $f$ is symmetric and strictly log-convex, (\ref{eq:singleleftder}) gives that $x_1+x_2 \leq 0$, while (\ref{eq:singlerightder}) gives that $x_1+x_2 \geq 0$. Therefore $x_1=-x_2$ and the interval is symmetric about the origin. Furthermore, for every given volume there is a unique symmetric interval.
\end{proof}

\begin{lemma}
\label{lem:nbubblelessthan3}
Consider $\R$ with a continuous density that is nonincreasing on $(-\infty,0]$ and nondecreasing on $[0,\infty)$. Let $M$ be the density minimum set where $f(x)=f(0)$. Consider two components of the same bubble in a perimeter-minimizing $n$-bubble.  Then the component on the right contains no points to the left of $M$ and some to the right of $M$. Similarly the component on the left contains no points right of $M$ and some left of $M$.
\end{lemma}

\begin{proof}
Let $M=[m_1,m_2]$ and the two components be $[a,b]$ and $[c,d]$ with $a<b<c<d$. For the right component, we need to show that $c \geq m_1$ and $d > m_2$.
Slide everything between $b$ and $c$ to the right, preserving volumes while decreasing the volume of $[c,d]$. 
If $d \leq m_2$, then the perimeter would not increase before $c$ reaches $d$, and it would decrease at that moment, a contradiction. Hence $d > m_2$.
If $c < m_1$, then while sliding $c$ up to $m_1$, the density never increases and decreases near $m_1$. So perimeter decreases, a contradiction.
A similar argument applies to the left component.
\end{proof}

\begin{prop}
\label{prop:2n-1}
On $\R$ with a continuous density that is nonincreasing on $(-\infty,0]$ and nondecreasing on $[0,\infty)$, a perimeter-minimizing $n$-bubble has at most $2n-1$ components.
\end{prop}

\begin{proof}
By Proposition \ref{prop:1d-contiguous}, all the components are contiguous. By Lemma \ref{lem:nbubblelessthan3}, each bubble has at most two components, so the $n$-bubble has at most $2n$ components. Moreover, if it has exactly $2n$ components, then the left component of any bubble lies to the left of the right component of every bubble. The right-most left component $L$ and the left-most right component $R$ meet at a point of minimum density. Denote the second components of the same bubbles by $L'$ and $R'$. They appear in the order $R', L, R, L'$. Now slide everything between $R'$ and $L$ to the right and everything between $R$ and $L'$ to the left, preserving volumes and not increasing perimeter, until either $L$ or $R$ disappears (all volume is contained in $L'$ or $R'$, respectively), reducing perimeter, a contradiction. Therefore the $n$-bubble has at most $2n-1$ components. 
\end{proof}

\begin{remark}
We suspect that the restriction to at most $2n-1$ components is sharp. In particular, we think that for some densities a perimeter-minimizing $n$-bubble for volumes $V_1 \ll V_2 \ll \cdots \ll V_n$ has $2n-1$ components: $V_1$ is centered on the origin flanked by $V_2/2$ on either side, which is flanked by $V_3/2$ on either side, and so on. 
Proposition \ref{lem:db-v1fixedv2large} proves this for $n=2$.
\end{remark}

\section{Double Bubbles on the Real Line}
\label{sect:db1d}
We now focus on the \emph{double} bubble and prove that perimeter minimizers are sometimes double intervals and sometimes triple intervals (Props. \ref{lem:db-equalvol} and \ref{lem:db-v1fixedv2large}).
This is consistent with Proposition \ref{prop:2n-1}, which states that a perimeter-minimizing double bubble has no more than 3 components. Theorem \ref{thm:db-1dmain} analyzes when each type occurs (see Fig. \ref{fig:orange}).

\begin{definition}
A \emph{double interval $(x_1,x_2,x_3)$} for prescribed volumes
 $V_1\leq V_2$ consists of two contiguous intervals $[x_1, x_2]$, $[x_2, x_3]$ of volumes $V_1$ and $V_2$, respectively, as in Figure \ref{fig:doubleinterval}. For a $C^1$ density $f$,
a double bubble is \emph{in equilibrium}
if it satisfies the consequence of perimeter minimization of Corollary \ref{prop:1d-equil}:
$$(\log f)'(x_1)+(\log f)'(x_2)+(\log f)'(x_3)=0.$$ The term also applies to the generalization to one-sided derivatives of Remark \ref{rem:onesidedderiv}.

The \emph{triple interval $(y_1,y_2)$} for prescribed volumes
$V_1 \leq V_2$ consists of three contiguous intervals, two of which flank the middle interval and enclose an equal volume, as in Figure \ref{fig:tripleinterval}. The middle interval is $[-y_1, y_1]$ and encloses volume $V_1$. The left interval is $[-y_2, -y_1]$ and the right interval is $[y_1, y_2]$, and each encloses volume $V_2/2$.

For a symmetric continuous, piecewise $C^1$ density, the triple interval is in equilibrium.
\end{definition}

\begin{figure}[ht]
  \includegraphics[width=0.7
\textwidth]{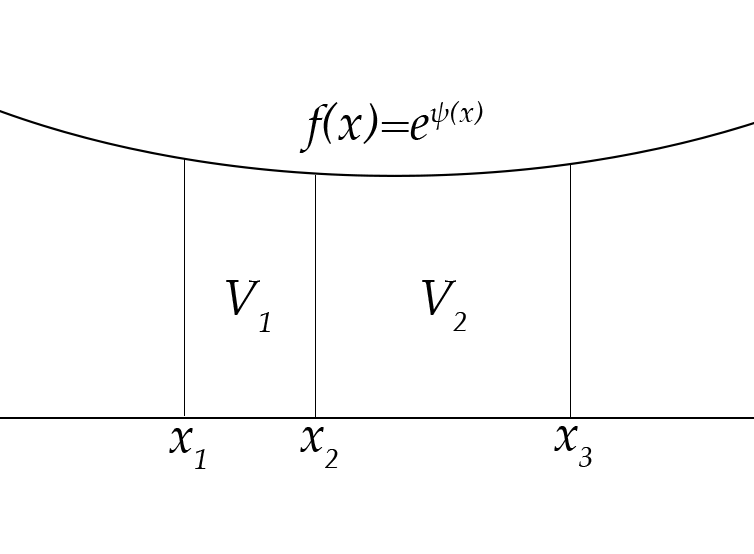}
\vspace{-10mm}
\caption{A double interval on the real line.}
\label{fig:doubleinterval}
\end{figure}

\begin{figure}[ht]
  \includegraphics[width=0.7
\textwidth]{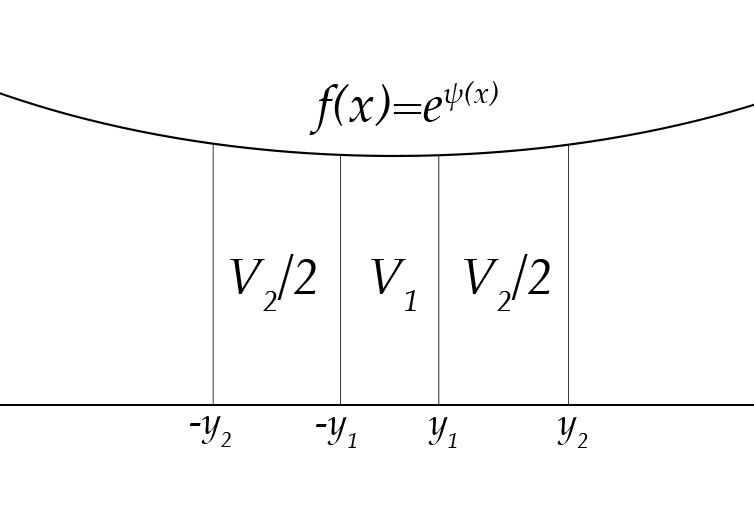}
\vspace{-10mm}
\caption{A triple interval on the real line.}
\label{fig:tripleinterval}
\end{figure}

\begin{figure}[h!]
  \includegraphics[width=0.7\textwidth]{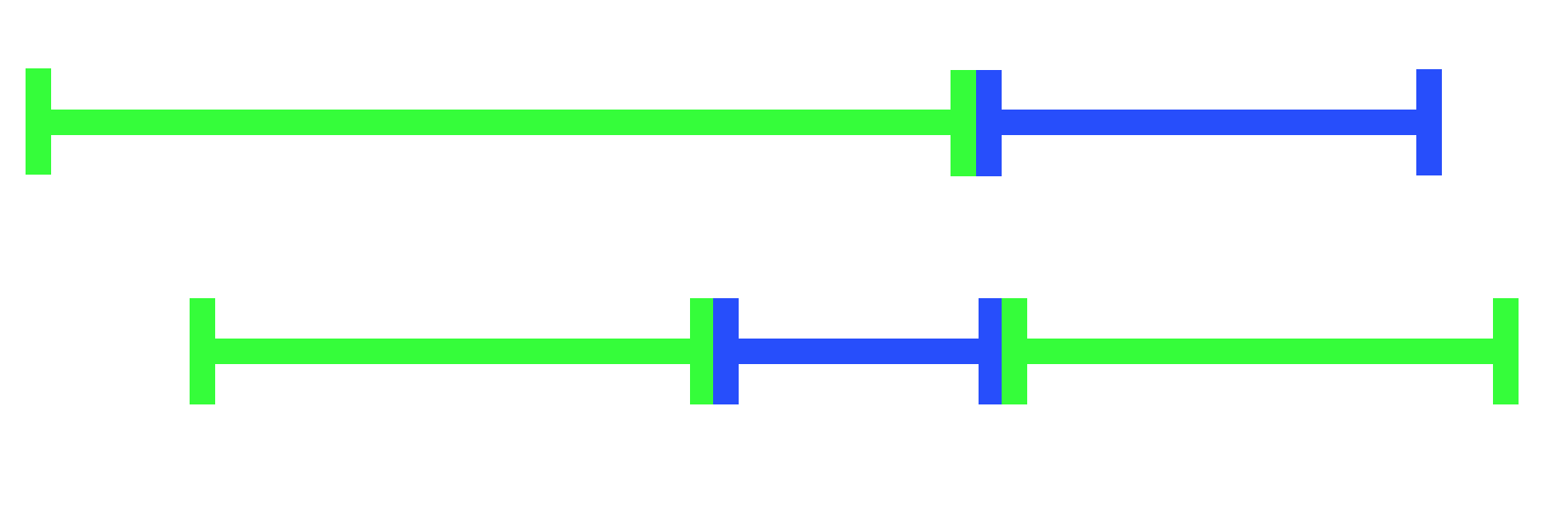}
\caption{Double and triple intervals in equilibrium on the real line.}
\label{fig:1ddb}
\end{figure}

Proposition \ref{prop:db-2or3} will characterize perimeter-minimizing double bubbles. First we show that a log-convex density can be considered as a convex density
in volume coordinate.

\begin{lemma}[Volume coordinate]
\label{lem:volcoord}
On $\R$ with density f, let
$$V = \int_0^x f.$$
Then $f$ is a log-convex function of $x$ if and only if $f$ is a convex function of $V$.
\end{lemma}

\begin{proof} 
The result follows from the fact that the one-sided derivatives satisfy
$$\frac{df}{dV} = \frac{df/dx}{dV/dx} = \frac{df/dx}{f} = \frac{d(\log f)}{dx}.$$
\end{proof}

The next lemma shows how to convert the volume coordinate back to
the positional coordinate.

\begin{lemma}
\label{lem:convertvolcoord}
On $\R$ with density f, let
$$V = \int_0^x f.$$
Then
$$x(V)=\int_0^V \frac{1}{f},$$
where $f$ is a function of $V$.
\end{lemma}

\begin{proof}
We have
$$\int_0^V \frac{1}{f}\, dV = \int_0^x \frac{1}{f} \frac{dV}{dx} \, dx
=\int_0^x \frac{1}{f} f \, dx = x.$$
\end{proof}

\begin{lemma}
\label{lem:db-2double}
On $\R$ with symmetric, strictly log-convex density,
for prescribed volumes $V_1\le V_2$,
if a perimeter-minimizing double bubble
has two components, then it is the unique double interval in equilibrium (up to reflection).
\end{lemma}

\begin{proof}
Let $f$ be the density.
By Proposition \ref{prop:1d-contiguous}, the intervals are contiguous,
so the double bubble must be a double interval $(x_1,x_2,x_3)$.
If $f$ is $C^1$, Corollary \ref{prop:1d-equil} implies that
the equilibrium condition
$$(\log f)'(x_1)+(\log f)'(x_2)+(\log f)'(x_3)=0.$$
holds. Moreover, assuming that the region on the left has volume $V_1$, this equation uniquely determines
the double interval: as $x_1$ moves, $x_2$ and $x_3$ also move
as strictly increasing functions of $x_1$. Hence the left-hand side
is a strictly increasing function of $x_1$ which tends to a negative value
as $x_1\to-\infty$ and tends to a positive value as $x_1\to\infty$.
The double interval satisfying the equation must therefore be unique.

If $f$ is not $C^1$, a similar argument applies using one-sided derivatives and Remark \ref {rem:onesidedderiv}.
\end{proof}

Proposition \ref{prop:constantdensity} shows that the strict log-convexity hypothesis in Lemma \ref{lem:db-2double} is necessary.

\begin{lemma}
\label{lem:db-3triple}
On $\R$ with symmetric, strictly log-convex density,
for prescribed volumes $V_1\le V_2$, if a perimeter-minimizing double bubble has three components, then it is the triple interval.
\end{lemma}

\begin{proof}
Let $V_1\leq V_2$ be the prescribed volumes.
By Corollary \ref{prop:1d-contiguous}, the intervals are contiguous. By applying Corollary \ref{prop:1d-equil} or Remark \ref {rem:onesidedderiv} to the middle interval we find that the middle interval is symmetric about the origin, and similarly the whole double bubble is also symmetric about the origin. 

Finally, it is optimal to place $V_1$ in the middle: since the total volume enclosed in the double bubble is the same regardless of which bubble is in the middle, we only need to examine the two inner boundary points. Since the perimeter is minimized when these points are nearest to the origin, the optimal choice is for the middle bubble
to enclose volume $V_1$.
Thus the perimeter-minimizing configuration is the triple interval.
\end{proof}

We can summarize the results of
Proposition \ref{prop:2n-1} and Lemmas \ref{lem:db-2double} and \ref{lem:db-3triple} in the following proposition.

\begin{prop}
\label{prop:db-2or3}
On $\R$ with symmetric, strictly log-convex density $f$,
for prescribed volumes $V_1\le V_2$,
a perimeter-minimizing double bubble is one of the following:
\begin{enumerate}[label = (\alph*)]
\item the unique double interval $(x_1,x_2,x_3)$ in equilibrium (up to reflection) or
\item the triple interval $(y_1,y_2)$.
\end{enumerate}
\end{prop}

\noindent See Figure \ref{fig:1ddb}.

\subsection*{Volume and perimeter relationships}
To understand better the transition from double to triple intervals, we examine volumes and perimeters more carefully. Let  $f$ be a symmetric, strictly log-convex, and $C^1$ density. For prescribed volumes $V_1 \leq V_2$,
let $P_2$ be the perimeter of the double interval in equilibrium
and $P_3$ the perimeter of the triple interval.
In volume coordinates (Lemma \ref{lem:volcoord}), we have
\begin{align*}
P_2 &= f(\widetilde{V})+f(\widetilde{V}+V_1)+f(\widetilde{V}+V_1+V_2), \\
P_3 &= 2\bracket{f\paren{\frac{V_1}{2}}+f\paren{\frac{V_1+V_2}{2}}},
\end{align*}
where $\widetilde{V}$ is the unique volume satisfying
the equilibrium condition for the double interval
\begin{equation}
\label{eq:equidoublevol}
f'(\widetilde{V})+f'(\widetilde{V}+V_1)+f'(\widetilde{V}+V_1+V_2)=0.
\end{equation}
Notice that the derivatives are in volume coordinates:
$$f'(V)=(\log f)'(x) \text{ where } V=\int_0^x f.$$
Taking derivatives of $P_2$ and $P_3$ yields
\begin{align}
P_2' &= f'(\widetilde{V}+V_1+V_2)V_2'-
f'(\widetilde{V})V_1',\label{eq:p2'}\\
P_3' &= f'\paren{\frac{V_1}{2}}V_1'+
f'\paren{\frac{V_1+V_2}{2}}(V_1'+V_2'),\label{eq:p3'}
\end{align}
where we used the equilibrium condition from
\eqref{eq:equidoublevol} for simplification.

\subsection*{Characterizations of when the double or triple interval
is perimeter minimizing}

\begin{figure}[!ht]
\includegraphics[width=0.8\textwidth]{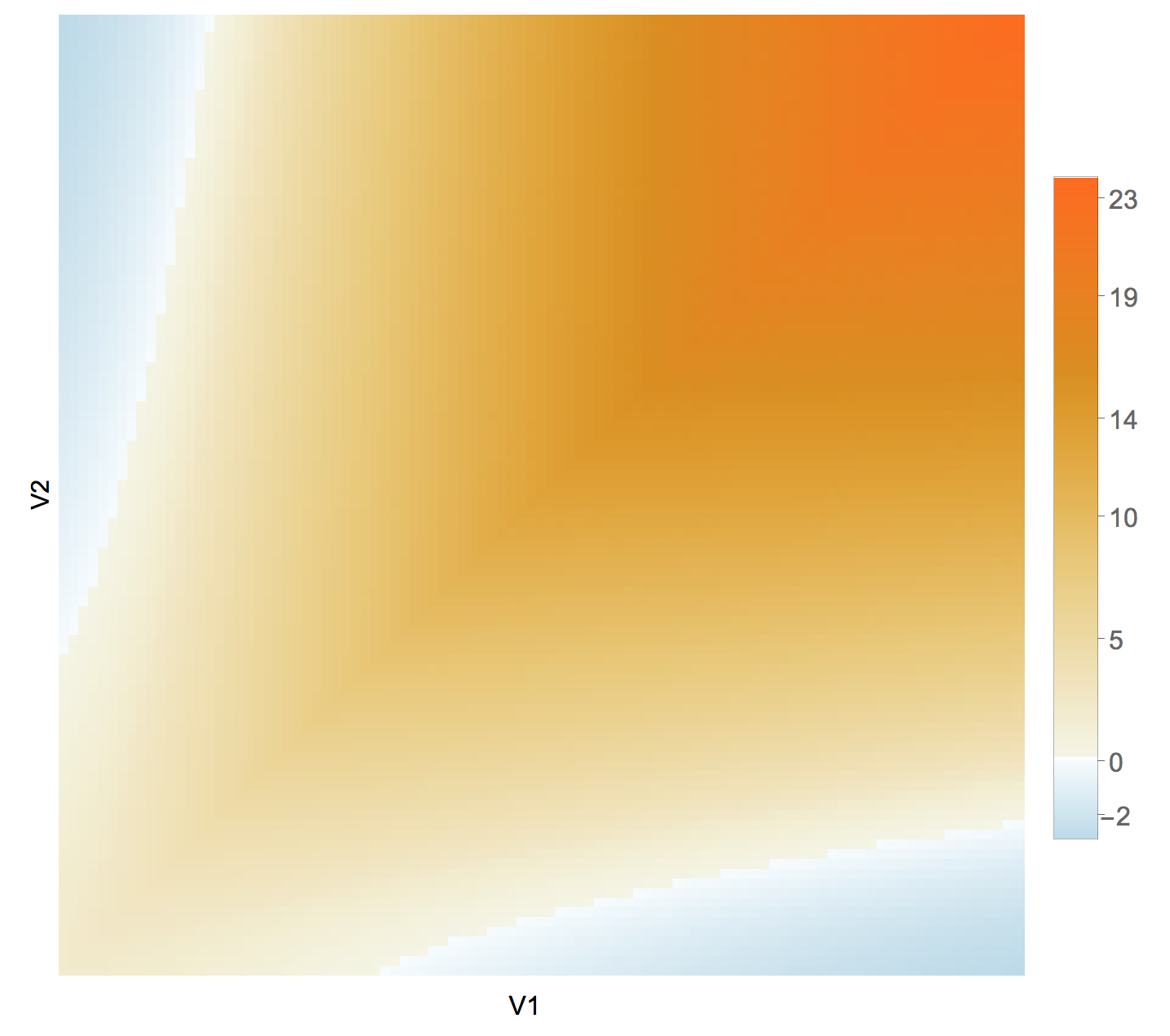}
\caption{The figure above is a numerical computation which represents the value of the perimeter difference $\mu(V_1, V_2)$ for Borell density $f(x)=e^{x^2}$. The orange color marks the region in which the double interval has lesser perimeter, the blue color represents the region in which the triple interval has lesser perimeter, and the white curve marks the tie point between the double and triple intervals. Computed in Mathematica.}
\label{fig:orange}
\end{figure}

\begin{definition}
For prescribed volumes $V_1\leq V_2$, let
$$\mu(V_1,V_2)=P_3-P_2$$
be the difference between the perimeter $P_3$ of the triple interval
and the perimeter $P_2$ of the double interval in equilibrium.
\end{definition}

By Proposition \ref{prop:db-2or3}, we obtain the following
characterization.
\begin{enumerate}[label = (\alph*)]
\item If $\mu(V_1,V_2)<0$,
then the perimeter-minimizing double bubble is uniquely the triple interval.
\item If $\mu(V_1,V_2)>0$,
then the perimeter-minimizing double bubble is uniquely the double interval
in equilibrium.
\item If $\mu(V_1,V_2)=0$,
then the perimeter-minimizing double bubble is either
the triple interval or the double interval in equilibrium.
\end{enumerate}

Let $f$ be a $C^1$ density.
Observe that by equations \eqref{eq:p2'} and \eqref{eq:p3'},
$\mu$ is a $C^1$ function with partial derivatives
\begin{align*}
\pdif{\mu}{V_1} &= f'\paren{\frac{V_1}{2}}+f'\paren{\frac{V_1+V_2}{2}}
+f'(\widetilde{V}),\\
\pdif{\mu}{V_2} &= f'\paren{\frac{V_1+V_2}{2}}-f'(\widetilde{V}+V_1+V_2).
\end{align*}

The remainder of this section investigates the behavior of $\mu$.

\begin{lemma}
\label{lem:db-bdvdouble}
On $\R$ with symmetric, strictly log-convex, $C^1$ density,
for prescribed volumes $V_1< V_2$,
$$-\frac{V_1+V_2}{2}<\widetilde{V}<-V_1.$$
\end{lemma}

\begin{proof}
Let $f$ be the density.
Consider the equilibrium condition 
$$f'(\widetilde{V})+f'(\widetilde{V}+V_1)+f'(\widetilde{V}+V_1+V_2)=0.$$
By Lemma \ref{lem:volcoord}, $f$ is convex in volume coordinate,
so the left-hand side is strictly increasing in $\widetilde{V}$.
At $\widetilde{V}=-(V_1+V_2)/2$, the left-hand side is negative,
while at $\widetilde{V}=-V_1$, the left-hand side is positive.
Hence the value of $\widetilde{V}$ that makes the left-hand side
vanish must lie inside the desired range.
\end{proof}

\begin{lemma}
\label{lem:db-muincdec}
Consider $\R$ with symmetric, strictly log-convex, $C^1$ density.
Given $V_2>0$,
$\mu$ is a strictly increasing function of $V_1 \leq V_2$.
Given $V_1>0$,
$\mu$ is a strictly decreasing function of $V_2 \geq V_1$.
\end{lemma}

\begin{proof}
Fix $V_2$. For $V_1<V_2$, we have
$$\pdif{\mu}{V_1} = f'\paren{\frac{V_1}{2}}+f'\paren{\frac{V_1+V_2}{2}}
+f'(\widetilde{V})>f'\paren{\frac{V_1}{2}}>0$$
due to Lemma \ref{lem:db-bdvdouble}.
Now fix $V_1$. For $V_2>V_1$, we have
$$\pdif{\mu}{V_2} = f'\paren{\frac{V_1+V_2}{2}}-f'(\widetilde{V}+V_1+V_2)
<0$$
due to Lemma \ref{lem:db-bdvdouble}.
\end{proof}

\begin{prop}
\label{lem:db-equalvol}
On $\R$ with symmetric, strictly log-convex, $C^1$ density,
for equal prescribed volumes $V_1=V_2$,
we have $\mu > 0$ (so the double interval is better).
\end{prop}

\begin{proof}
For $V_1=V_2$, we have $\widetilde{V}=-V_1$. So
$$P_2=2f(V_1)+f(0)<2f\paren{\frac{V_1}{2}}+2f(V_1)=P_3.$$
\end{proof}

\begin{prop}
\label{lem:db-v1fixedv2large}
On $\R$ with symmetric, strictly log-convex, $C^1$ density $f$ such that $(\log f)'$ is unbounded,
given $V_1>0$, we have $\mu < 0$
for large $V_2\geq V_1$ (so the triple interval is better).
\end{prop}

\begin{proof}
Fix $V_1$. For $V_2$ large, we need to show that
$$P_2 = f(\widetilde{V})+f(\widetilde{V}+V_1)+f(\widetilde{V}+V_1+V_2) >
2\bracket{f\paren{\frac{V_1}{2}}+f\paren{\frac{V_1+V_2}{2}}} = P_3.$$
By convexity of $f$ in volume coordinate (Lemma \ref{lem:volcoord}),
$$f(-\widetilde{V})+f(\widetilde{V}+V_1+V_2)\geq 2f\paren{\frac{V_1+V_2}{2}}.$$
Notice that $f(-\widetilde{V})=f(\widetilde{V})$ by symmetry of $f$.
So we need to show that, for $V_2$ large,
$$f(\widetilde{V}+V_1) > 2f\paren{\frac{V_1}{2}}.$$
It suffices to show that $\widetilde{V}\to -\infty$ as $V_2\to\infty$.

From the equilibrium condition 
$$f'(\widetilde{V})+f'(\widetilde{V}+V_1)+f'(\widetilde{V}+V_1+V_2)=0,$$
as $V_2\to\infty$, if $\widetilde{V}$ does not become very small,
then the leftmost two terms stay bounded, while the rightmost term
goes to infinity because $f'(V)$ is unbounded, which is a contradiction.
Hence $\widetilde{V}\to -\infty$ as $V_2\to\infty$.
\end{proof}

\begin{remark}
Propositions \ref{prop:constantdensity} and \ref{prop:densityeabsx} show that the hypothesis of \emph{strict} log-convexity in Proposition \ref{lem:db-v1fixedv2large} is necessary.
Moreover, the following example shows that the hypothesis that $(\log f)'$ is unbounded is needed.
\end{remark}

\begin{example}
Consider the density in volume coordinate $f(V)=\abs{V}+e^{-\abs{V}}$.
Notice that $f$ is $C^1$ and $f(V)$ is strictly convex, but $f'(V)$ is bounded. For fixed $V_1$, as $V_2 \to \infty$, it can be computed that $\widetilde{V} \to -\log (1+e^{V_1})$. We can then check that
$$\mu \to 2V_1 - \log(1+e^{V_1}) + 2e^{-V_1/2} - 1 > 0$$
for all $V_1 > 0$. Since $\mu$ is decreasing in $V_2$ (Lemma \ref{lem:db-muincdec}), $\mu > 0$ and so the double interval is better for all $V_1$ and $V_2$.
\end{example}

\begin{lemma}
\label{lem:db-v2small}
On $\R$ with symmetric, strictly log-convex, $C^1$ density,
for small $V_2>0$, $\mu > 0$ for all $V_1\leq V_2$ (so the double interval is better).
\end{lemma}

\begin{proof}
For $V_2$ small, by Lemma \ref{lem:db-bdvdouble}, $\widetilde{V}$
is also small in magnitude.
Hence every density term that contributes to $P_2$ and $P_3$
is close to $f(0)$. Thus for $V_2$ small, $P_2$ is close to $3f(0)$
while $P_3$ is close to $4f(0)$, so that $P_2<P_3$.
\end{proof}

\begin{theorem}
\label{thm:db-1dmain}
On $\R$ with symmetric, strictly log-convex, $C^1$ density $f$ such that $(\log f)'$ is unbounded,
given $V_1>0$, there is a unique $V_2=\lambda(V_1)$
such that the double interval in equilibrium and the triple interval
tie.
For $V_2>\lambda(V_1)$, the perimeter-minimizing
double bubble is uniquely the triple interval.
For $V_2<\lambda(V_1)$, the perimeter-minimizing
double bubble is uniquely the double interval
in equilibrium.
Moreover, $\lambda$ is a strictly increasing $C^1$ function
that tends to a positive limit as $V_1\to 0$.
\end{theorem}

\begin{proof}
Fix $V_1$. By Lemma \ref{lem:db-muincdec}, $\mu$
is a strictly decreasing function of $V_2$.
By Lemma \ref{lem:db-equalvol}, $\mu > 0$ for
$V_2=V_1$. By Lemma \ref{lem:db-v1fixedv2large},
$\mu < 0$ for large $V_2$.
These together imply that there is a unique $V_2=\lambda(V_1)$
such that $\mu = 0$ and that
$\mu > 0$ for $V_2<\lambda(V_1)$ and
$\mu < 0$ for $V_2>\lambda(V_1)$.
Thus $\lambda$ determines the perimeter-minimizing double bubbles
as in the theorem statement.

Observe that $\mu$ is a $C^1$ function with partial derivative
$\del \mu/\del V_2 < 0$
at points $(V_1,\lambda(V_1))$, by Lemma \ref{lem:db-muincdec}.
So by the implicit function theorem, $\lambda$ is a $C^1$ function.

We now show that
$\lambda$ is strictly increasing.
Suppose not. Then there are $V_1<V_1^*$ with
$\lambda(V_1)\geq\lambda(V_1^*)$.
By Lemma \ref{lem:db-muincdec},
$$0=\mu(V_1,\lambda(V_1))<\mu(V_1^*,\lambda(V_1))\leq
\mu(V_1^*,\lambda(V_1^*))=0,$$
a contradiction. Note the term in the middle makes sense
because $V_1^*\leq \lambda(V_1)$.

Finally we show that $\lambda$ tends to a positive limit
as $V_1\to 0$. By Lemma \ref{lem:db-v2small}, there is $v>0$
such that $\mu > 0$ for all $V_1\leq V_2\leq v$,
so $\lambda(V_1)\neq v$ for all $V_1$.
Thus because $\lambda$ is strictly increasing,
it tends to a limit which is at least as big as $v$
as $V_1\to 0$.
\end{proof}


It is an interesting question what happens in the case where $(\log f)'$ is bounded. In our follow-up note \cite{So2}, we show that the tie function $\lambda$ still exists but only for $V_1 \in (0,V_0)$ for some ``blowup time'' $0\leq V_0 \leq \infty$, and $\lambda \to \infty$ as $V_1 \to V_0$.

\section{Non-Strictly Log-Convex Densities}
\label{sect:nonstrict}

Section \ref{sect:nonstrict} considers some densities which are symmetric, piecewise $C^1$,
and log-convex, but not strictly log-convex.

We investigate the following densities:
\begin{enumerate}[label = (\roman*)]
\item The constant density $f(x)=c$ (Prop. \ref{prop:constantdensity}).
\item The exponential density $f(x)=e^{\abs{x}}$ (Prop. \ref{prop:densityeabsx}).
\item The smoothed-out exponential density
\begin{equation}
\label{eqn:smoothabsx}
f(x)=
\begin{cases}
e^{x^2}\quad &\text{for }\abs{x}<a\\
e^{a(2\abs{x}-a)} \quad &\text{for }\abs{x}\geq a
\end{cases}
\end{equation}
\noindent with $a > 2 \sqrt{\log 2}$ 
(Prop. \ref{prop:mainsmoothedout}).
\end{enumerate}

\noindent For the constant density, every double interval is perimeter minimizing. For the exponential density, a perimeter-minimizing double bubble is a double interval with the middle point at the origin. For the smoothed-out exponential density, the triple interval appears for $V_1$ small and $V_2$ large.

For non-strictly log-convex densities, there may be a continuum of double intervals (or triple intervals) in equilibrium with the prescribed volumes. Lemma \ref{lem:nonstrictequilibria} shows that all such are perimeter minimizing among double intervals (among triple intervals).

\begin{lemma}
\label{lem:nonstrictequilibria}
On $\R$ with symmetric log-convex density f, among n-bubbles of prescribed volumes and fixed combinatorial type, every equilibrium is perimeter minimizing.
\end{lemma}

\begin{proof} 
We may assume that $f$ is not constant, since the result is easy in that case.
If we switch to volume coordinate (Lemma \ref{lem:volcoord}), then
$f$ is a convex function of volume $V$. Since $f$ is symmetric, convex, and nonconstant, it goes to infinity in both directions. Hence a minimizer exists. Since every minimizer is in equilibrium, it remains to show that every equilibrium has equal perimeter.

Represent an $n$-bubble as a tuple of volume coordinates of endpoints of its components. Then given two $n$-bubbles in equilibrium $B_1$ and $B_2$, on the straight line between them, the volume of each component varies linearly.
Since the volume of each bubble is equal at $B_1$ and $B_2$, the volume of each bubble must be constant along this straight line. Hence all $n$-bubbles along this straight line have the prescribed volumes. Let $P(t)$ denote the perimeter of the $n$-bubble $(1-t)B_1+tB_2$, $t \in [0,1]$, on this straight line.
Then $P$ is convex because $f$ is convex, and the one-sided derivatives $P'_R(0) \geq 0$ and $P'_L(1) \leq 0$ because $B_1$ and $B_2$ are in equilibria. It follows that $P$ is constant, and so $B_1$ and $B_2$ have equal perimeter.
\end{proof}

\begin{remark}
Lemma \ref{lem:nonstrictequilibria} reduces the search for a perimeter-minimizing double bubble to any double 
interval or triple interval in equilibrium. We can pick a perimeter-minimizing triple interval 
to be the triple interval symmetric about the origin.

The proof shows that the set of (perimeter-minimizing) equilibria is a finite-dimensional cell, convex in the $V$ coordinates.

The hypothesis that $f$ be symmetric, used for the existence of a minimizer, is not necessary. If $f$ approaches but never reaches a limit in either direction, equilibria do not exist (because the derivative of perimeter is negative as you slide the $n$-bubble in that direction), and the result holds trivially. Otherwise minimizers exist.
\end{remark}

\begin{remark}
\label{rem:nonstrictformula}
By Lemma \ref{lem:nonstrictequilibria}, in volume coordinates (Lemma \ref{lem:volcoord}),
the perimeters $P_2$ of double and $P_3$ of triple intervals in equilibrium are still given by
\begin{align*}
P_2 &= f(\widetilde{V})+f(\widetilde{V}+V_1)+f(\widetilde{V}+V_1+V_2), \\
P_3 &= 2\bracket{f\paren{\frac{V_1}{2}}+f\paren{\frac{V_1+V_2}{2}}},
\end{align*}
where $\widetilde{V}$ is any volume satisfying
the equilibrium condition
$$f'(\widetilde{V})+f'(\widetilde{V}+V_1)+f'(\widetilde{V}+V_1+V_2)=0.$$
\end{remark}

We now consider some specific non-strict log-convex densities.

\begin{prop}
\label{prop:constantdensity}
On $\R$ with density $f(x)=c$, any double interval enclosing the prescribed volumes is perimeter minimizing.
\end{prop}

\begin{proof}
A double interval has perimeter $3c$. Any other configuration has perimeter
at least $4c$. Therefore a double interval is perimeter minimizing.
\end{proof}

\begin{prop}
\label{prop:densityeabsx}
On $\R$ with density $f(x)=e^{\abs{x}}$, the perimeter-minimizing double bubble is the double interval with the middle perimeter point at the origin, unique up to reflection across the origin.
\end{prop}

\begin{proof}
The density in volume coordinate (Lemma \ref{lem:volcoord}) is
$$f(V)=1+\abs{V}.$$
We first consider a double interval in equilibrium.
The function $f$ is $C^1$ everywhere except at the origin with $f'(V)=1$ or $-1$.
So the equilibrium condition (Rem. \ref{rem:nonstrictformula}) cannot be satisfied
unless one boundary point is at the origin. By Remark \ref{rem:onesidedderiv},
the leftmost boundary point cannot be at the origin because the sum of the left derivatives
$$\sum f'_L=1$$
would be positive. Similarly the rightmost boundary point cannot be at the origin.
Hence the middle boundary point is at the origin. So the double interval in equilibrium is
unique up to reflection across the origin.

Now we can compare the perimeters of double and triple intervals in equilibrium:
$$P_2=f(V_1)+f(0)+f(V_2)=V_1+V_2+3<2V_1+V_2+4=P_3.$$
Therefore the perimeter-minimizing double bubble is the double interval in equilibrium.
\end{proof}

We now consider the smoothed-out exponential density \eqref{eqn:smoothabsx}.

\begin{lemma}
\label{lem:smoothf2int}
Consider $\mathbb{R}$ with the smoothed-out exponential density \eqref{eqn:smoothabsx}.
Let $V_1 \leq V_2$ be prescribed volumes. If 
$$V_1 + V_2\leq \int_{0}^a e^{x^2}\,dx,$$
then the perimeter-minimizing double bubbles are the same double and triple intervals as for the Borell density $f(x) = e^{x^2}$ of Proposition \ref{prop:db-2or3}.
\end{lemma}

\begin{proof}
By Proposition \ref{prop:1d-contiguous}, a perimeter-minimizing double bubble consists of contiguous intervals.
Observe that one of these intervals must contain the origin, as otherwise the whole double bubble can be shifted
towards the origin and the perimeter will decrease.
By the upper bound on $V_1+V_2$, the whole bubble is contained in $[-a,a]$.
In this interval, the density is identical to the Borell density.
By Proposition \ref{prop:db-2or3}, a perimeter-minimizing double bubble for the Borell density
also lies in this interval for the prescribed volumes.
So perimeter-minimizing double bubbles for the two densities are identical.
\end{proof}

\begin{lemma}
\label{lem:smoothv1large}
Consider $\mathbb{R}$ with the smoothed-out exponential density \eqref{eqn:smoothabsx}.
Let $V_1 \leq V_2$ be prescribed volumes, with
$$V_1\geq\int_{-a}^a e^{x^2}\,dx.$$
Then the double interval in equilibrium has the middle boundary point at the origin,
and a triple interval in equilibrium has boundary points
$$-y_2 < -y_1\leq - a < 0 < a \leq y_1' < y_2',$$ 
free up to the volume constraints.  
\end{lemma}

\begin{proof}
Observe that $f$ is $C^1$ and that
$$ (\log f)'(x) = 
\begin{cases}
-2a \quad &\text{for } x \leq - a\\
2x &\text{for } \abs{x} < a\\
2a &\text{for } x \geq a.
\end{cases}$$
By Corollary \ref{prop:1d-equil}, the sum of the derivatives of the log of the density at the boundary points of a double interval in equilibrium must equal zero.
We claim that the middle boundary point is $0$. If the middle boundary point were less than zero, then the sum of the derivatives of the log of the density would be negative. If the middle boundary point were greater than $0$, then the sum of the derivatives of the log of the density would be positive. Therefore the middle boundary point is $0$, as asserted.

For a triple interval in equilibrium, let the boundary points be $-y_2<-y_1<y_1'<y_2'$.
If $y_1'$ is less than $a$, then the sum of the derivatives of the log of the density at the boundary points is negative.  If $-y_1$ is greater than $-a$, then the sum of the derivatives of the log of the density at the boundary points is positive.
Therefore $-y_1\leq -a<a\leq y_1'$,
and the sum of the derivatives of the log of the density equals zero whenever this holds.
So the claim holds.
\end{proof}

\begin{lemma}
\label{lem:smoothv1smallv2large}
Consider $\mathbb{R}$ with the smoothed-out exponential density \eqref{eqn:smoothabsx}. Let $V_1 \leq V_2$ be prescribed volumes, where
$$V_1 < \int_0^{a}e^{x^2}\,dx, \quad V_1+V_2 \geq \int_{-a}^a e^{x^2} \,dx.$$
Then a double interval in equilibrium has perimeter points $x_1,x_2,x_3$, where 
$$-a \leq x_1<x_2 = -a -x_1 \leq 0<a \leq x_3,$$
up to reflection across the origin.
A perimeter-minimizing triple interval has perimeter points
$$-y_2 \leq -a < -y_1 < 0 < y_1 < a \leq y_2',$$  
where $[-y_1,y_1]$ has volume $V_1$ and $y_2$ and $y_2'$ are free up to the volume constraint.
\end{lemma}

\begin{proof}
Let a double interval in equilibrium have boundary points $x_1 < x_2 < x_3$,
and assume that the left interval has volume $V_1$.
By Corollary \ref{prop:1d-equil},
\begin{equation}
\label{eq:smoothv1smallv2largedouble}
(\log f)'(x_1)+(\log f)'(x_2)+(\log f)'(x_3)=0.
\end{equation}
If $x_1 < -a$, then $x_2<0$, and so \eqref{eq:smoothv1smallv2largedouble}
implies that $(\log f)'(x_3)>2a$, which is impossible. Hence $x_1 \geq -a$.
Because of the bound on $V_1+V_2$, $x_3\geq a$. 
Now by \eqref{eq:smoothv1smallv2largedouble}, $(\log f)'(x_2)\leq 0$, so $x_2 \leq 0$. 
So \eqref{eq:smoothv1smallv2largedouble} reduces to $x_1 + x_2 + a = 0,$ or $x_2 = -a-x_1$. 
Then the inequality stated in the proposition holds.
The resulting double interval is unique because there is only one $x_1$ such that $[x_1,-a-x_1]$ has volume $V_1$,
and $x_3$ is determined from $x_1$.

For a perimeter-minimizing triple interval,
denote the boundary points by
$-y_2< -y_1 < y_1'<y_2'$.
By Corollary \ref{prop:1d-equil},
\begin{align}
(\log f)'(-y_1)+(\log f)'(y_1')&=0 \label{eqn:55midinterval}\\
(\log f)'(-y_2)+(\log f)'(y_2')&=0, \nonumber
\end{align}
so either $-y_2,y_2'\in(-a,a)$ with $y_2=y_2'$ or
$-y_2\leq -a$ and $y_2'\geq a$.
By the bound on $V_1+V_2$, the latter must be the case.
Similarly,
either $-y_1,y_1'\in(-a,a)$ with $y_1=y_1'$ or
$-y_1\leq -a$ and $y_1'\geq a$. 
To determine which is the case, we must first determine whether $[-y_1,y_1']$ encloses volume $V_1$ or $V_2$.

We claim that $[-y_1,y_1']$ must enclose volume $V_1$.
By the volume restriction on $V_1$ and (\ref{eqn:55midinterval}), if $[-y_1,y_1']$ has volume $V_1$, then $-y_1,y_1'\in(-a,a)$ with $y_1=y_1'$. If $[-y_1,y_1']$ encloses volume $V_2$, 
then the magnitudes of $y_1$ and $y_1'$---and hence the perimeter $f(-y_1)+f(y_1')$---will be greater than 
when $[-y_1,y_1']$ encloses volume $V_1$,
so in order for the triple interval to be perimeter minimizing,
$[-y_1,y_1']$ must enclose volume $V_1$.
Then $y_1'=y_1\in(0,a)$, and
therefore the boundary points satisfy the inequality
in the statement of the proposition.
\end{proof}

The following proposition is our most interesting example, which shows that the appearance of the triple interval can depend on $V_1$. The particular value $2\sqrt{\log 2}$ is just for our convenience.

\begin{prop}
\label{prop:mainsmoothedout}
Consider $\mathbb{R}$ with the smoothed-out exponential density \eqref{eqn:smoothabsx}: 
$$f(x)=
\begin{cases}
e^{x^2}\quad &\text{for }\abs{x} < a\\
e^{a(2\abs{x}- a)} \quad &\text{for }\abs{x}\geq a
\end{cases}
$$
with $a > 2 \sqrt{\log 2}$. Let $V_1 \leq V_2$ be prescribed volumes. If 
\begin{equation}
\label{eqn:v1geqhalf}
V_1 \geq\int_{-a}^a e^{x^2}\,dx,
\end{equation} 
then the perimeter-minimizing double bubble is the double interval in equilibrium for all $V_2$ (up to reflection). 
If $V_1$ is small, then for $V_2$ close to $V_1$ the perimeter-minimizing double bubble is the double interval in equilibrium
and for $V_2$ large a perimeter-minimizing double bubble is a triple interval.
\end{prop}

\noindent See Figure \ref{fig:smoothorange}.

\begin{figure}[!ht]
\includegraphics[width=0.7\textwidth]{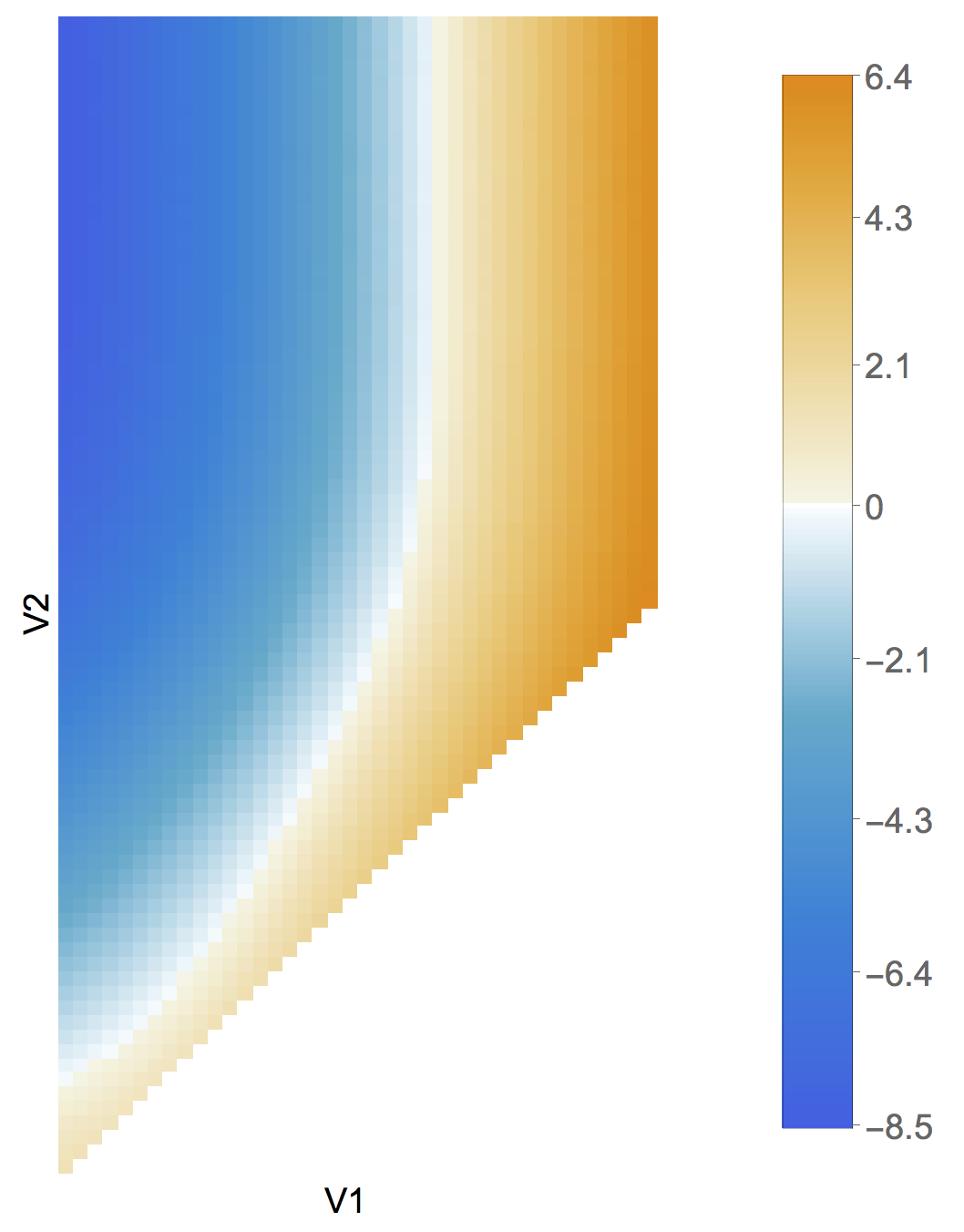}
\caption{The figure above is a numerical computation which represents the value of the perimeter difference $\mu(V_1, V_2)$ for the smoothed out exponential density \eqref{eqn:smoothabsx}. The orange color marks the region in which the double interval has less perimeter, the blue color represents the region in which the triple interval has less perimeter, and the white curve marks the tie point between the double and triple intervals. The plot indicates that the tie curve asymptotes to fixed value of $V_1$. Computed in Mathematica.}
\label{fig:smoothorange}
\end{figure}

\begin{proof}
By Proposition \ref{prop:1d-contiguous}, a perimeter-minimizing double bubble consists of finitely many contiguous intervals. By Proposition \ref{prop:2n-1}, a perimeter-minimizing double bubble consists of two or three such intervals.

First consider the case when $V_1$ is small $V_2$ is close to $V_1$. We assume that 
$$V_1+V_2 \leq \int_0^a e^{x^2}\,dx.$$
By Lemma \ref{lem:smoothf2int},
a perimeter-minimizing double bubble is identical to
the one for the Borell density $f(x)=e^{x^2}$ with the same prescribed volumes.
By Lemma \ref{lem:db-equalvol}, it is the double interval in equilibrium for $V_2$ close to $V_1$.

Now suppose (\ref{eqn:v1geqhalf}) holds.  
By Lemma \ref{lem:smoothv1large}, the boundary points of the double interval in equilibrium are $x_1,0,x_3$, where $x_1\leq-a$ and $x_3 \geq a$,
and the boundary points of a triple interval in equilibrium are $-y_2,-y_1,y_1',y_2'$, where we may choose $-y_2=x_1$ and hence $y_2' = x_3$.  Thus we only need to compare the inner boundary points.
Because $f(-y_1)+f(y_1') > f(0) = 1$, a double interval in equilibrium has less perimeter than a triple interval in equilibrium.
Therefore the perimeter-minimizing double bubble is the double interval in equilibrium for all $V_2$.

Finally, suppose that $V_1$ is small and $V_2$ is large.
By Lemma \ref{lem:smoothv1smallv2large}, a double interval in equilibrium has boundary points $x_1,x_2,x_3$, where $x_2$ is close to $-a/2$,
and a perimeter-minimizing triple interval has boundary points $-y_2,-y_1,y_1,y_2'$, where $y_1$ is close to zero, $-y_2\leq -a$, and $y_2'\geq a$.
Observe that the single bubbles $[-y_2,y_2']$ and $[x_1,x_3]$
have volume $V_1+V_2$ and that by Corollary \ref{prop:1d-equil}
the first single bubble is in equilibrium.
Thus the perimeter of the outer boundary points of a perimeter-minimizing triple interval is less than or equal to the perimeter of the outer boundary points of a double interval in equilibrium.  So it remains to examine the perimeter from the inner boundary points. Observe that 
$$f\paren{-\frac{a}{2}} = e^{{a^2}/4}>e^{(2\sqrt{\log2})^2/4} = 2
=2f(0),$$
since $a>2\sqrt{\log 2}$.
Because $x_2$ is close to $-a/2$ and $y_1$ is close to 0,
the perimeter from the inner boundary points of a perimeter-minimizing triple interval is less than the perimeter from the inner boundary points of a double interval in equilibrium. Then the total perimeter for a perimeter-minimizing triple interval is less than the total perimeter for a double interval in equilibrium. Therefore for $V_1$ small and $V_2$ large, a perimeter-minimizing double bubble is a triple interval.
\end{proof}

\begin{conj}
Consider $\mathbb{R}$ with the smoothed-out exponential density \eqref{eqn:smoothabsx}. Let $V_1 \leq V_2$ be prescribed volumes.
Then there exists $V_0>0$ such that for $V_1\geq V_0$,
a perimeter-minimizing double bubble is always a double interval.
For $0<V_1<V_0$, there is a unique $V_2=\lambda(V_1)$ such that
double intervals and triple intervals tie.
For $V_2>\lambda(V_1)$, a perimeter-minimizing double bubble
is a triple interval.
For $V_2<\lambda(V_1)$, a perimeter-minimizing double bubble
is a double interval.
Moreover, $\lambda$ is a strictly increasing $C^1$ function
that tends to a positive limit as $V_1\to 0$
and tends to infinity as $V_1\to V_0$.
\end{conj}

\section{Bounds on the Tie Points}
\label{sect:bounds}
We put some bounds on the growth of the function $V_2=\lambda(V_1)$,
where the double and triple intervals tie,
as defined in Theorem \ref{thm:db-1dmain}.
From this point onwards we assume that $f$ is a symmetric, strictly log-convex, $C^1$ density such that $(\log f)'$ is unbounded.

One of the main results of this section is that $\lambda(V_1)/V_1\to\infty$
as $V_1\to\infty$ for densities $f=e^\psi$ such that $\psi_{xx}/\psi_x$ is bounded
for $x$ large (Cor. \ref{cor:lbpositional}).
As in Remark \ref{rem:tieintersecttwice},
this implies that a line through the origin can intersect
the tie curve $(V_1,\lambda(V_1))$ more than once.
For the Borell density, this is illustrated in Figure \ref{fig:tieintersection} of the Introduction.

\begin{lemma}
For prescribed volumes $V_1 \leq V_2$, let $(x_1,x_2,x_3)$ be the double interval in equilibrium (with the left interval enclosing volume $V_1$)
and $(y_1,y_2)$ be the triple interval.
Fix $V_1$. As $V_2$ increases,
$y_2$ and $x_3$ increase, while
$x_1$ and $x_2$ decrease.
\end{lemma}

\begin{proof}
It is easy to see that $y_2$ increases. As $V_2$ increases,
$x_3$ can be moved to the right to accommodate the increased volume.
The double interval is no longer in equilibrium, with
$$\sum_{i=1}^3 \psi_x(x_i)>0.$$
In order to be in equilibrium, the double interval must shift left.
This implies that $x_1$ and $x_2$ decrease.
To show that $x_3$ increases, note that in equilibrium,
$$\sum_{i=1}^3 \psi_x(x_i)=0.$$
So because $x_1$ and $x_2$ decrease, $x_3$ must increase so that
the sum remains zero.
\end{proof}

The next lemma shows that the density in volume coordinate grows at least linearly
but is approximately linear.

\begin{lemma}
\label{lem:fapproxlinear}
Let $f$ be a symmetric, strictly log-convex, $C^1$ density on $\R$.
In volume coordinate, there is $c>0$ such that for large $V$,
$$f(V)\geq cV.$$
Moreover,
$$\int_0^\infty \frac{1}{f(V)}$$
diverges.
\end{lemma}

In particular, this implies that although $f$ grows asymptotically at least as fast as $V$,
it cannot grow asymptotically faster than $V^c$ for any $c>1$.

\begin{proof}
Because $f'$ is strictly increasing (Lemma \ref{lem:volcoord}), there is $\eps > 0$ such that
for large $V$, $f'>\eps$. Hence for large $V$, $f(V)\geq (\eps/2)V$.

The integral diverges because as we take $V\to\infty$
in the formula in Lemma \ref{lem:convertvolcoord}, it must be the case that $x(V)\to\infty$.
\end{proof}

The next lemma shows that the rate of growth of the volume of the interval $[0,x]$
is on the order of $f(x)/\psi_x$ for typical densities.
Notice that the hypothesis $\psi_x^2 \geq M\psi_{xx}$ for large $x$ is mild
and holds for all densities we are interested in, namely $e^{x^n}$ and $e^{e^x}$.
In fact, it holds if the sign of $\psi_x^2-M\psi_{xx}$ changes only a finite number
of times, because we can easily check that a function satisfying
$\psi_x^2<M\psi_{xx}$ blows up in finite time.

\begin{lemma}[Fundamental Bounding Lemma]
\label{lem:fundbound}
Let $f=e^\psi$ be a symmetric, strictly log-convex, $C^1$ density on $\R$.
For $x>0$, define
$$V=\int_0^x f.$$
Then
$$\liminf_{x\to\infty} \frac{V}{f/\psi_x} \geq 1.$$
Furthermore, if $f$ is $C^2$ and $M>1$ is such that
$\psi_x^2 \geq M\psi_{xx}$ for $x$ large, then
$$\limsup_{x\to\infty} \frac{V}{f/\psi_x} \leq \frac{M}{M-1}.$$
In particular, if $\psi_x^2/\psi_{xx}\to\infty$ as $x\to\infty$, then
$$\lim_{x\to\infty} \frac{V}{f/\psi_x} = 1.$$
\end{lemma}

\begin{proof}
In volume coordinate, we can rewrite the quatity in question via Lemma \ref{lem:volcoord} as
$$\frac{V}{f/\psi_x}=\frac{V}{f(V)/f'(V)}=\frac{Vf'(V)}{f(V)}.$$
Because $f'$ is nondecreasing,
$$f(V)-f(0)=\int_0^V f' \leq Vf'(V),$$
so that
$$\frac{Vf'(V)}{f(V)}\geq \frac{f(V)-f(0)}{f(V)}\to 1$$
as $V\to\infty$, implying the desired lower bound.

Now suppose $\psi_x^2 \geq M\psi_{xx}$ for $x$ large.
In volume coordinate, $\psi_x=f'(V)$
and
$$\psi_{xx}=\frac{df'(V)}{dx}=
\frac{df'(V)}{dV}\frac{dV}{dx} = f''(V)f(V).$$
So $f'^2/(ff'')\geq M$ for large $V$. Now for $V$ large
$$\paren{\frac{f}{f'}}'=\frac{f'^2-ff''}{f'^2}\geq 1-\frac{1}{M},$$
so that, for constants $c$ and $c_1$,
$$\frac{f}{f'}=c+\int_1^V \paren{\frac{f}{f'}}'\geq c_1+\frac{M-1}{M}V.$$
Therefore
$$\limsup_{V\to\infty} \frac{Vf'}{f}\leq \frac{M}{M-1},$$
as desired.
\end{proof}

For lower bounds, we start by showing that, although $V_2^*\geq V_1$ by definition,
it never approaches the line $V_2=V_1$.

\begin{prop}
On $\R$ with a symmetric, strictly log-convex, $C^1$ density $f$ such that $(\log f)'$ is unbounded,
given $V_1>0$, let $V_2^*$ be the unique value of $V_2$
such that the double interval in equilibrium and the triple interval tie.
Then $V_2^*>2V_1$ for all $V_1>0$.
\end{prop}

\begin{proof}
In volume coordinate (Lemma \ref{lem:volcoord}), we have
\begin{multline*}
P_2 = f(\widetilde{V})+f(\widetilde{V}+V_1)+f(\widetilde{V}+V_1+V_2^*) \\
= 2f\paren{\frac{V_1}{2}}+2f\paren{\frac{V_1+V_2^*}{2}}=P_3 ,
\end{multline*}
where $\widetilde{V}$ is the unique value such that the double interval is in equilibrium.
If we use $-(V_1+V_2^*)/2$ in place of $\widetilde{V}$, then the resulting double interval
has perimeter greater than or equal to the original perimeter:
$$P_2\leq 2f\paren{\frac{V_1+V_2^*}{2}}+f\paren{\frac{V_2^*-V_1}{2}}.$$
Therefore
$$f\paren{\frac{V_2^*-V_1}{2}}\geq 2f\paren{\frac{V_1}{2}}>f\paren{\frac{V_1}{2}},$$
so that, because $f$ is increasing, $V_2^*>2V_1$.
\end{proof}

We now show that for slow-growing density $f$, 
specifically where $Vf''(V)$ is bounded in volume coordinate,
the function $\lambda$ grows superlinearly.
We do not know whether this hypothesis is sharp.

\begin{prop}
\label{prop:lambdatoinfty}
Let $f$ be a symmetric, strictly log-convex, $C^2$ density on $\R$.
Suppose that, in volume coordinate, $f'(V)$ is unbounded and $Vf''(V)$ is bounded.
Given $V_1>0$, let $V_2^*$ be the unique value of $V_2$
such that the double interval in equilibrium and the triple interval tie.
Then $V_2^*/V_1\to\infty$ as $V_1\to\infty$.
\end{prop}

\begin{proof}
By Lemma \ref{lem:fapproxlinear}, there is $c>0$ such that for large $V_1$,
$2f(V_1/2)>2cV_1>cV_1+f(0)$. So
\begin{align*}
&cV_1\leq 2f\paren{\frac{V_1}{2}}-f(0)=\mu(V_1,V_1)-\mu(V_1,V_2^*) \\
&=- \int_{V_1}^{V_2^*} \pdif{\mu}{V_2}\, dV_2 
= \int_{V_1}^{V_2^*} f'(\widetilde{V}+V_1+V_2)-f'\paren{\frac{V_1+V_2}{2}} \, dV_2 \\
& \leq \int_{V_1}^{V_2^*} f'(V_2)-f'(V_1) \, dV_2
\leq \int_{V_1}^{V_2^*} f'(V_2^*)-f'(V_1) \, dV_2 \\
&=(V_2^*-V_1)(f'(V_2^*)-f'(V_1)) \leq V_2^*(f'(V_2^*)-f'(V_1)),
\end{align*}
where the third line follows because $f'$ is nondecreasing and $\widetilde{V}<-V_1$
due to Lemma \ref{lem:db-bdvdouble}. Therefore, for $V_1$ large,
\begin{equation}
\label{eq:lb-main}
f'(V_2^*)-f'(V_1) \geq \frac{c}{V_2^*/V_1}.
\end{equation}

Because $f'$ is unbounded, $f(V)/V\to\infty$ as $V\to\infty$,
and so the constant $c$ in \eqref{eq:lb-main} can be taken arbitrarily large.
By hypothesis on the growth of $f$, there is $c_1>0$ such that $f''(V)\leq c_1/V$. So
$$f'(V_2^*)-f'(V_1)=\int_{V_1}^{V_2^*} f''
\leq c_1 \int_{V_1}^{V_2^*} \frac{1}{V} = c_1 \log \frac{V_2^*}{V_1}.$$
Hence by \eqref{eq:lb-main}
$$\frac{V_2^*}{V_1}\log \frac{V_2^*}{V_1} \geq \frac{c}{c_1}.$$
As $V_1\to\infty$, $c$ can be taken arbitrarily large, and so $V_2^*/V_1\to\infty$,
as desired.
\end{proof}

\begin{remark}
Instead of $f$ being $C^2$, it suffices to assume that $f$ is $C^1$ and that there is $c>0$ such that $f'(V)-c \log V$ is eventually nonincreasing.
\end{remark}

The following corollary translates the hypothesis of Proposition \ref{prop:lambdatoinfty}
to the positional coordinate. It says that the conclusion of Proposition \ref{prop:lambdatoinfty}
holds for well-behaved densities that can grow as fast as $\exp(e^{cx})$, because
this is the density where $\psi_{xx}/\psi_x$ is constant.
We do not know whether this condition is sharp.

\begin{cor}
\label{cor:lbpositional}
Let $f=e^\psi$ be a symmetric, strictly log-convex, $C^2$ density on $\R$,
where $\psi_x$ is unbounded and $\psi_{xx}/\psi_x$ is bounded for $x$ large.
Given $V_1>0$, let $V_2^*$ be the unique value of $V_2$
such that the double interval in equilibrium and the triple interval tie.
Then $V_2^*/V_1\to\infty$ as $V_1\to\infty$.

Moreover, there is $c>0$ such that
for any $\eps > 0$, for $V_1$ large,
$$V_2^*\geq V_1 (\psi_x f^{-1}(cV_1))^{1-\eps},$$
where $f^{-1}$ is the inverse function of the density in the \emph{positional} coordinate.
\end{cor}

\begin{proof}
Because $\psi_x$ is unbounded and $\psi_{xx}/\psi_x$ is bounded for $x$ large, $\psi_x^2/\psi_{xx}\to\infty$
as $x\to\infty$, so the Fundamental Bounding Lemma \ref{lem:fundbound} applies.
We have
$$Vf''(V)=V\frac{df'(V)}{dV}=V\frac{d\psi_x}{dx}\frac{dx}{dV}=\frac{V\psi_{xx}}{f}.$$
By Lemma \ref{lem:fundbound}, $V$ is on the same order as $f/\psi_x$ for $x$ large, so
$Vf''(V)$ is on the same order as $\psi_{xx}/\psi_x$, which is bounded.
Therefore the conclusion of the first part follows from Proposition \ref{prop:lambdatoinfty}.

For the second part, we use the method of Proposition \ref{prop:lambdatoinfty} to arrive at
\begin{equation}
\label{eq:concretelb}
\frac{V_2^*}{V_1}\log \frac{V_2^*}{V_1} \geq c\frac{f(V_1/2)}{V_1}
\end{equation}
for some constant $c>0$. By Lemma \ref{lem:fundbound}, the quantity on the right-hand side
is on the same order as $\psi_x(x(V_1/2))$, where $x$ is the function that converts from
volume to positional coordinate. By Lemma \ref{lem:fapproxlinear}, there is a constant $c_1>0$
such that
$$f(x(V))\geq c_1V$$
for $V$ large. Because $f$ is strictly increasing,
$$x(V)\geq f^{-1}(c_1V).$$
So \eqref{eq:concretelb} becomes
$$\frac{V_2^*}{V_1}\log \frac{V_2^*}{V_1} \geq c\psi_x f^{-1}(c_1V/2),$$
which implies the conclusion.
\end{proof}

\begin{remark}
\label{rem:tieintersecttwice}
Along a line $V_2 = rV_1$, the perimeter minimizer may change from a triple interval to a double interval and back to a triple interval, as numerically plotted for the Borell density $f(r)=\exp \paren{r^2}$ in Figure \ref{fig:tieintersection} of the introduction, where the line $V_2 = 10V_1$ intersects the curve of tie points in two places.
Indeed, whenever $\psi_x$ is unbounded and $\psi_{xx}/\psi_x$ is bounded
for $x$ large, Corollary \ref{cor:lbpositional} implies that $\lambda(V_1)/V_1\to\infty$ as $V_1\to\infty$,
while by Theorem \ref{thm:db-1dmain}, $\lambda$ tends to a positive limit as $V_1\to 0$, so $\lambda(V_1)/V_1\to\infty$ as $V_1\to 0$. Thus $\lambda(V_1)/V_1$ must assume some value twice in the interval $(0,\infty)$.
\end{remark}

The next two corollaries follow immediately from Corollary \ref{cor:lbpositional}. Corollary \ref{cor:lowerboundBorelllambda} gives a lower bound on the tie function $\lambda$
for the Borell density. Numerics suggest that the bound is not sharp: Figure \ref{fig:tieintersection} of the introduction suggests that $\lambda$
grows approximately quadratically for the Borell density.

\begin{cor}
\label{cor:lowerboundBorelllambda}
On $\R$ with the Borell density $f(x)=e^{x^2}$,
given $V_1>0$, let $V_2^*$ be the unique value of $V_2$
such that the double interval in equilibrium and the triple interval tie.
Then for all $\eps>0$,
$$V_2^*\geq V_1(\log V_1)^{1/2-\eps}$$
for $V_1$ large.
\end{cor}

\begin{cor}
\label{cor:lbeex}
On $\R$ with the density $f(x)=e^{e^x}$,
given $V_1>0$, let $V_2^*$ be the unique value of $V_2$
such that the double interval in equilibrium and the triple interval tie.
Then for all $\eps>0$,
$$V_2^*\geq V_1(\log V_1)^{1-\eps}$$
for $V_1$ large.
\end{cor}

The following proposition gives an upper bound for the tie points.
Note that the hypothesis $\psi^2 \geq M\psi_x$ for $x$ large is mild.
In fact, $\psi^2 < M\psi_x$ cannot hold for all $x$ large,
since this inequality implies that the $\psi$ blows up in finite time.

\begin{prop}
\label{prop:lambdaupbd}
Let $f=e^\psi$ be a symmetric, strictly log-convex, $C^1$ density on $\R$,
where $\psi_x$ is unbounded and $\psi^2 \geq M\psi_x$ for $x$ large.
Given $V_1>0$, let $V_2^*$ be the unique value of $V_2$
such that the double interval in equilibrium and the triple interval tie.
Then, for constants $c_1>0$ and $c_2$,
$$V_2^* \leq c_1\exp\paren{\psi\psi_x^{-1}(2\psi_x A)}$$
for $V_1$ large, where
$$A=\psi^{-1}(\log V_1+2\log\log V_1+c_2).$$
\end{prop}

\begin{proof}
Let $V_2=V_2^*$. Let the double interval in equilibrium be
$(x_1,x_2,x_3)$ and the triple interval be $(y_1,y_2)$.
We have
$$f(x_1)+f(x_2)+f(x_3)=P_3=P_2=2f(y_1)+2f(y_2).$$
By Proposition \ref{prop:1d-single}, the interval $[-y_2,y_2]$ is the best
single bubble for volume $V_1+V_2$, so
$$f(x_1)+f(x_3)\geq 2f(y_2).$$
Thus
$$f(x_2)\leq 2f(y_1),$$
which is equivalent to
\begin{equation}
\label{eq:lu-psix2}
\psi(x_2)\leq \psi(y_1)+\log 2.
\end{equation}

By the Fundamental Bounding Lemma \ref{lem:fundbound},
for $x$ large,
\begin{equation}
\label{eq:lu-f}
f(x)\leq 2V(x)\psi_x(x) \leq \frac{2}{M}V(x)\psi(x)^2
=\frac{2}{M}V(x)(\log f)^2,
\end{equation}
because $\psi^2\geq M\psi_x$ for $x$ large.
For $x$ large, $\log f \leq f^{1/4}$, so $\eqref{eq:lu-f}$ implies that
$f\leq cV^2$ for a constant $c$. Using this on the right-hand side
of $\eqref{eq:lu-f}$ gives $f\leq cV(\log V)^2$ for a new constant $c$
(we allow $c$ to change from line to line).
Hence \eqref{eq:lu-psix2} becomes
\begin{align}
\psi(x_2)&\leq \log f(y_1)+\log 2
\leq \log V(y_1)+2\log\log V(y_1)+c \nonumber \\
&\leq \log V_1+2\log\log V_1 + c,
\label{eq:lu-x2est}
\end{align}
because $V(y_1)=V_1/2$.
So $\abs{x_2}\leq A$, where $A$ is defined as in the proposition statement.

We now estimate $x_1$. We have
$$V_1=\int_{|x_2|}^{|x_1|} f \geq
\frac{1}{\psi_x(|x_1|)}\int_{|x_2|}^{|x_1|} \psi_x f
=\frac{f(x_1)-f(x_2)}{\psi_x(|x_1|)}.$$
So
\begin{align}
f(x_1)&\leq f(x_2)+\psi_x(|x_1|)V_1
\leq f(x_2)+\frac{1}{M}\psi(x_1)^2V_1 \nonumber \\
&=f(x_2)+\frac{1}{M}(\log f(x_1))^2V_1 \nonumber \\
&\leq cV_1(\log V_1)^2 + \frac{1}{M}(\log f(x_1))^2V_1,
\label{eq:lu-fx1}
\end{align}
by \eqref{eq:lu-x2est}.
For $x$ large,
$\log f \leq f^{1/4}$, so by $\eqref{eq:lu-fx1}$,
$$f(x_1)\leq cV_1(\log V_1)^2+\frac{1}{M}f(x_1)^{1/2}V_1.$$
Solving gives $f(x_1)\leq cV_1^2$.
Applying this to the right-hand side of $\eqref{eq:lu-fx1}$ yields
$$f(x_1)\leq cV_1(\log V_1)^2.$$
Hence
\begin{equation}
\label{eq:lu-x1est}
\psi(x_1)=\log f(x_1)\leq \log V_1 + 2\log\log V_1 + c,
\end{equation}
which implies that $\abs{x_1}\leq A$ where $A$ is defined as in the
proposition statement.

Therefore, for $V_1$ large,
\begin{align*}
V_2^*&\leq 2\int_0^{x_3} f \leq 2\int_1^{x_3}f
\leq \frac{2}{\psi_x(1)}\int_1^{x_3} \psi_x f\\
&=\frac{2(f(x_3)-f(1))}{\psi_x(1)}
\leq cf(x_3)=c\exp \psi(x_3)\\
&=\exp\psi\psi_x^{-1}(\psi_x(x_3))
=\exp\psi\psi_x^{-1}\paren{\psi_x(|x_1|)+\psi_x(|x_2|)}\\
&\leq \exp\psi\psi_x^{-1}(2\psi_xA),
\end{align*}
by \eqref{eq:lu-x2est} and \eqref{eq:lu-x1est}, because
$f'=f\psi_x$ and $\psi_x(x_3)=\psi_x(|x_1|)+\psi_x(|x_2|)$
due to the equilibrium.
\end{proof}

The following corollaries compute the upper bounds explicitly
for the densities $e^{x^2}$ and $e^{e^x}$.

\begin{cor}
\label{cor:upperboundBorelllambda}
On $\R$ with the Borell density $f(x)=e^{x^2}$,
given $V_1>0$, let $V_2^*$ be the unique value of $V_2$
such that the double interval in equilibrium and the triple interval tie.
Then for all $\eps>0$,
$$V_2^*\leq V_1^{4+\eps}$$
for $V_1$ large.
\end{cor}

\begin{proof}
Since $\psi=x^2$ and $\psi_x=2x$,
$\psi^2\geq \psi_x$ for $x$ large, so 
Proposition \ref{prop:lambdaupbd} applies.
Therefore
\begin{align*}
V_2^*&\leq c_1\exp{4(\log V_1+2\log\log V_1+c_2)} \\
&=cV_1^4(\log V_1)^8 \leq V_1^{4+\eps}
\end{align*}
for any $\eps>0$ and $V_1$ large,
where $c_1$, $c_2$ and $c$ are constants.
\end{proof}

\begin{cor}
On $\R$ with the density $f(x)=e^{e^x}$,
given $V_1>0$, let $V_2^*$ be the unique value of $V_2$
such that the double interval in equilibrium and the triple interval tie.
Then for all $\eps>0$,
$$V_2^*\leq V_1^{2+\eps}$$
for $V_1$ large.
\end{cor}

\begin{proof}
Since $\psi=\psi_x=e^x$, $\psi^2\geq \psi_x$ for $x$ large,
so Proposition \ref{prop:lambdaupbd} applies.
Therefore
\begin{align*}
V_2^*&\leq c_1\exp{2(\log V_1+2\log\log V_1+c_2)} \\
&=cV_1^2(\log V_1)^4 \leq cV_1^{2+\eps}
\end{align*}
for any $\eps>0$ and $V_1$ large,
where $c_1$, $c_2$ and $c$ are constants.
\end{proof}

\section{Higher Dimensions}
\label{sect:higher}
In $\R^N$ with radial density going to infinity,
for any $n$ given volumes, a perimeter-minimizing $n$-bubble exists by an argument (\cite{BDKS}, Rmk. 3.3) after Morgan and Pratelli (\cite{MP}, Thm. 3.3).

In this section we use numerical techniques to determine the shape and surface area of potentially perimeter-minimizing double bubbles in higher dimensions with Borell density $e^{r^2}$ as in Figure \ref{fig:brakketwothree} of the Introduction. The code used for these computations can be found at \url{https://github.com/arjunkakkar8/doublebubble}.

To do so, we use the software Brakke's Surface Evolver \cite{Br}. Starting with an initial shape, the surface evolver iteratively minimizes the energy associated with that configuration by moving its pieces while maintaining the values of constraints defined on the configuration.

To examine the double bubble in space with Borell density, we define the initial configuration of two adjacent cubes (squares in the plane). Then the energy of the system is defined as the weighted perimeter of the cubes. Next the weighted volume of the cubes is calculated by using the divergence theorem. With boundary $B$, the weighted volume is
$$V_B=\oint_{B} \mathbf{F}\cdot\mathbf{n}, \text{ where } \nabla\cdot\mathbf{F}=e^{\mathbf{r}^2}.$$

\noindent Note that the choice of $\mathbf{F}$ for which $\nabla\cdot\mathbf{F}=e^{\mathbf{r}^2}$ is not unique. We used
?$$\mathbf{F}=\bigg(e^{y^2+z^2}\int e^{x^2} dx, \quad 0, \quad 0 \bigg).$$
Since there is no closed form for the integral, the vector field is evaluated by using a series expansion. Care is taken to use sufficiently many terms so that within the relevant radius,
the error from the approximation is negligible compared to the
4-digit precision of the Evolver.
Then the system is evolved down the energy gradient while fixing the weighted volumes. The final state that the system converges to for the case of 2D and of 3D is depicted in Figure \ref{fig:brakketwothree}.

In closing, we conjecture that some of the behavior on $\R^1$ will recur in higher dimensions.

\begin{conj} In $\mathbb{R}^N$ with a smooth, radial, log-convex density, a perimeter-minimizing double bubble is either
\begin{enumerate}[label=(\roman*)]
\item the bubble inside a bubble (e.g. for $V_1$ small and $V_2$ large), or
\item the standard double bubble (e.g. for $V_2$ close to $V_1$).
\end{enumerate}
\end{conj}

\bibliographystyle{alpha}

\begin{thebibliography}{10}

\bibitem[Ba]{Ba} [Ba] V. Bayle (2004). Propri\'{e}t\'{e}s de concavit\'{e} du profil isop\'{e}rim\'{e}trique et applications, PhD thesis, Institut Joseph Fourier, Grenoble.

\bibitem[BDKS]{BDKS} [BDKS] 
Eliot Bongiovanni, Alejandro Diaz, Arjun Kakkar, Nat Sothanaphan. Isoperimetry in surfaces of revolution with density, Missouri J. Math. Sci. 30 (2018), no. 2, 150--165. \url{https://arxiv.org/abs/1709.06040}.


\bibitem[Br]{Br} [Br] Kenneth Brakke. The surface evolver, \url{http://facstaff.susqu.edu/brakke/evolver/evolver.html}.

\bibitem[BH]{BH} [BH] Serguei G. Bobkov, Christian Houdr\'{e} (1997). Some connections between isoperimetric and Sobolev-type inequalities, Mem. Amer. Math. Soc. 129, no. 616, viii+111.

\bibitem[Ch]{Ch} [Ch] Gregory R. Chambers (2015). Proof of the Log-Convex Density Conjecture. J. Eur. Math. Soc., to appear.


 



\bibitem[M]{M} [M] Frank Morgan (2016). {Geometric Measure Theory: A Beginner's Guide}, Academic press.



\bibitem[MN]{MN} [MN] Emanuel Milman, Joe Neeman (2018). 
The Gaussian double-bubble conjecture, preprint. \url{https://arxiv.org/abs/1801.09296}.

\bibitem[MP]{MP} [MP] Frank Morgan, Aldo Pratelli (2013). Existence of isoperimetric regions in $\R^n$ with density, Ann. Glob. Anal. Geom. 43, 331-365.

\bibitem[MG]{MG} [MG] Ivor McGillivray (2017). An isoperimetric inequality in the plane with a log-convex density, preprint. \url{https://arxiv.org/abs/1612.07052}.


\bibitem[RCBM]{RCBM} [RCBM] C\'esar Rosales, Antonio Ca\~nete, Vincent Bayle, Frank Morgan (2008). {On the isoperimetric problem in Euclidean space with density}, Calc. Var. 31, 27-46.

\bibitem[So1]{So1} [So1] Nat Sothanaphan (2018). 1D Triple Bubble Problem with Log-Convex Density. \url{https://arxiv.org/abs/1805.08377}.

\bibitem[So2]{So2} [So2] Nat Sothanaphan. Double Bubbles on the Line with Log-convex Density $f$ with $(\log f)'$ Bounded, Missouri J. Math. Sci. 30 (2018), no. 2, 166--175. \url{ https://arxiv.org/abs/1807.02661}.



\end{thebibliography}

\vspace{1cm}
\noindent Eliot Bongiovanni\\
Michigan State University \\
\url{eliotbonge@gmail.com }

\vspace{.7cm}
\noindent Leonardo Di Giosia\\
Rice University \\
\url{lsd2@rice.edu}

\vspace{.7cm}
\noindent Alejandro Diaz\\
University of Maryland, College Park \\
\url{diaza5252@gmail.com}

\vspace{.7cm}
\noindent Jahangir Habib\\
Williams College \\
\url{jih1@williams.edu }

\vspace{.7cm}
\noindent Arjun Kakkar\\
Williams College \\
\url{ak23@williams.edu }

\vspace{.7cm}
\noindent Lea Kenigsberg\\
Columbia University in the City of New York\\
\url{lk2720@columbia.edu}

\vspace{.7cm}
\noindent Dylanger Pittman\\
Williams College\\
\url{dsp1@williams.edu}

\vspace{.7cm}
\noindent Nat Sothanaphan\\
Massachusetts Institute of Technology \\
\url{natsothanaphan@gmail.com}

\vspace{.7cm}
\noindent Weitao Zhu\\
Williams College \\
\url{wz1@williams.edu }

\end{document}